\documentclass[a4paper]{amsart}
\usepackage{amssymb}
\usepackage{xfrac} 
\usepackage{comment}
\usepackage{tikz}
\usetikzlibrary{cd}
\usepackage[
    backend=biber,
    style=numeric,
    citestyle=nature,
    maxcitenames=9,
    maxbibnames=9,
    sorting=nyt,
    natbib=true,
    url=false, 
    doi=false,
    eprint=true
]{biblatex}
\newtheorem{theorem}{Theorem}[section] 
\newtheorem{lemma}[theorem]{Lemma}
\newtheorem{proposition}[theorem]{Proposition}

\newtheorem{corollary}[theorem]{Corollary}

\theoremstyle{definition}

\newtheorem{definition}[theorem]{Definition}
\newtheorem{remark}[theorem]{Remark}
\newtheorem{example}[theorem]{Example}

\newcommand{\RR}{\mathbb{R}}

\newcommand{\NN}{\mathbb{N}}
\DeclareMathOperator{\supp}{supp}
\DeclareMathOperator{\vol}{vol}
\DeclareMathOperator{\Val}{Val}
\newcommand{\calH}{\mathcal{H}}
\newcommand{\calK}{\mathcal{K}}

\newcommand{\wt}{\widetilde}
\DeclareMathOperator{\Sym}{Sym}

\newcommand{\sL}{\mathring{\mathrm L}}

\def\dottedbox{\tikz\node[draw=black,dotted] {\vphantom{.}};}
\title{Dually Lorentzian Polynomials}
\author{Julius Ross}
\author{Hendrik S\"uss}
\author{Thomas Wannerer}
\address{Mathematics, Statistics and Computer Science, University of Illinois at Chicago, 851 S. Morgan Street, 322 Science and Engineering Offices, Chicago, IL~60607, United States}
\email{juliusro@uic.edu}
\address{Friedrich-Schiller-Universit\"at Jena, Fakult\"at f\"ur Mathematik und Informatik, Institut f\"ur Mathematik, Ernst-Abbe-Platz 2, 07743 Jena, Germany}
\email{hendrik.suess@uni-jena.de}
\address{Friedrich-Schiller-Universit\"at Jena, Fakult\"at f\"ur Mathematik und Informatik, Institut f\"ur Mathematik, Ernst-Abbe-Platz 2, 07743 Jena, Germany}
\email{thomas.wannerer@uni-jena.de}

\DeclareMathOperator{\GL}{GL}
\begin{document}

\begin{abstract}
We introduce and study a notion of dually Lorentzian polynomials, and show that if $s$ is non-zero and dually Lorentzian then the operator $$s(\partial_{x_1},\ldots,\partial_{x_n}):\mathbb R[x_1,\ldots,x_n] \to \mathbb R[x_1,\ldots,x_n]$$ preserves (strictly) Lorentzian polynomials.  From this we conclude that any theory that admits a mixed Alexandrov-Fenchel inequality also admits a generalized Alexandrov-Fenchel inequality involving dually Lorentzian polynomials.  As such we deduce generalized Alexandrov-Fenchel inequalities for mixed discriminants, for integrals of K\"ahler classes, for mixed volumes, and in the theory of valuations.
\end{abstract}
\maketitle
\section{Introduction} 

Lorentzian polynomials, recently introduced by Br{\"a}nd{\'e}n-Huh \cite{BrandenHuh}, have coefficients that satisfy a form of log-concavity,  and have been used to prove, reprove, and conjecture various combinatorial statements coming from convex geometry,  representation theory, and  the theory of matroids.

Accepting that being Lorentzian is a useful concept, it is natural to ask for linear operators that preserve this property, for one may consider such operators  as analogous to operations that preserve log-concavity of sequences.   To this end we introduce the following notion.

\begin{definition} Let $s\in \mathbb R[x_1,\ldots,x_n]$ be a polynomial of multi-degree at most $(\kappa_1,\ldots,\kappa_n)$.  We say that $s$ is \emph{dually Lorentzian} if the polynomial
$$ s^{\vee}(x_1,\ldots,x_n):= N(x_1^{\kappa_1}\cdots x_n^{\kappa_n} s(x_1^{-1},\ldots,x_n^{-1}))$$
is Lorentzian.
\end{definition}
Here $N$ is the normalization defined to be the linear operator such that $N(x^{\alpha}) = (\alpha !)^{-1} x^\alpha$.   This notion is well defined in that being dually Lorentzian is independent of choice of the $\kappa_i$.  Examples of dually Lorentzian polynomials include Schur and Schubert polynomials (Theorem  \ref{thm:schubert}) but there are many other examples of dually Lorentzian polynomials including some that are not Schur positive (Example \ref{ex:duallylorentzianbutnotschurpositive}). \vspace{0mm}

Our main statement about such polynomials is the following:

\begin{theorem}[Theorem \ref{thm:operator-criterion}, Theorem \ref{thm:strictly-lorentzian}]\label{thm:intropreserve}
Suppose that $s\in \mathbb R[x_1,\ldots,x_n]$ is non-zero.  Then $s$ is dually Lorentzian if and only if the operator
$$\partial_s  = s(\partial_{x_1},\ldots,\partial_{x_n}) : \mathbb R[x_1,\ldots,x_n] \to  \mathbb R[x_1,\ldots,x_n]$$
preserves the property of being (strictly) Lorentzian. 
\end{theorem}

We will also show that this statement extends to the class of $C$-Lorentzian polynomials of Br\"and\'en and Leake \cite{BrandenLeake:CLorentzian} (Theorem \ref{thm:invarianceC}).   These seemingly simple statements quickly give rise to a number of interesting applications.  As advocated by Huh \cite{MR3966524}, a theme running through many areas of mathematics is the so-called ``K\"ahler package" whose axioms are essentially a form of Poincar\'e duality, hard Lefschetz theorem and the Hodge-Riemann bilinear relations.  From any such package one obtains a Lorentzian polynomial, analogous to the volume polynomial of a collection of ample line bundles on a projective variety, or to the volume polynomial of a collection of convex bodies.   As such one obtains for each dually Lorentzian polynomial some log-concavity type statements coming from this K\"ahler package.   Said another way, whenever one has a theory that admits a mixed Alexandrov-Fenchel (AF) inequality, the corresponding volume polynomial will be Lorentzian and we can deduce a generalized AF inequality involving any dually Lorentzian polynomial.

We illustrate this idea with the following applications, with the expectation that many more are possible.

\begin{theorem}[=Corollary \ref{cor:duallorentizanlogconcave}, Log-concavity of Derived Polynomials]\label{thm:introcombinatorics}
Let $s$ be dually Lorentzian of degree $d$ and define polynomials $s^{(j)}(x_1,\ldots,x_n)$ by requiring $$s(x_1+u,\ldots,x_n+u) = \sum_{j=0}^d s^{(j)}(x_1,\ldots,x_n) u^j.$$  Then for fixed $t_1,\ldots,t_n\in \mathbb R_{\ge 0}$ the map
$$j\mapsto s^{(j)}(t_1,\ldots,t_n)$$
is log-concave.
\end{theorem}

\begin{theorem}[=Theorem \ref{thm:mixeddiscriminants},\! Generalized AF inequality for Mixed Discriminants] \label{thm:intromixedd} Let $V$ be a $d$-dimensional real vector space and $s$ be a non-zero dually Lorentzian polynomial of degree $d-2$ in $n$ variables.  Suppose $q_1,\ldots, q_{n}$, $q$ are positive definite quadratic forms on $V$ and  $p$ is an arbitrary quadratic form.  Then the mixed discriminant $D$ satisfies
	\begin{equation} D(p,q, s(q_1,\ldots,q_n))^2 \geq D(p,p, s(q_1,\ldots,q_n)) D(q,q, s(q_1,\ldots,q_n))\end{equation}
 and  equality holds if and only if $p$ is proportional to $q$.
\end{theorem}

\begin{theorem}[=Theorem \ref{thm:hodge-riemann}, Generalized AF inequality for K\"ahler classes]\label{thm:introkahler}
	Let $Y$ be a smooth complex manifold of dimension $d$ and $\omega_1, \ldots, \omega_m,\alpha$ be K\"ahler classes on $Y$.  Then for every non-zero dually Lorentzian polynomial $s$ of degree $d-2$ and all $\beta\in H_\RR^{1,1}(Y)$
	$$\left(\int_Y \alpha \beta s(\omega_1,\ldots,\omega_m) \right)^2 \ge \int_Y \alpha^2 s(\omega_1,\ldots,\omega_m)\int_Y \beta^2 s(\omega_1,\ldots,\omega_m)$$
and equality holds if and only if $\alpha$ is proportional to $\beta$.
\end{theorem}


In the next statement we are able to characterize
 the equality case for convex bodies $K$ in the class $C^2_+$ which requires $\partial K$ to be a strictly curved $C^2$ submanifold of $\mathbb R^n$. 

\begin{theorem}[=Theorem \ref{thm:genAF}, Generalized AF inequality for mixed volumes] 
 Let $K,L,C_1, \ldots, C_n$ be   convex bodies in $\RR^d$.  Then  
 \[
 V(K,L,s(C_1, \ldots, C_n))^2 \geq V(K,K,s(C_1, \ldots, C_n))V(L,L,s(C_1, \ldots, C_n))
 \]  
 for every non-zero dually Lorentzian polynomial $s$ of degree $d-2$.
Moreover, if $K$, $L$, $C_1,\ldots, C_{n}$ are of class $C_+^2$, then equality holds if and only if $K$ and $L$ are homothetic.
\end{theorem}

The conditions that we require to characterize the case of equality in  Theorem~\ref{thm:genAF} are not optimal. We hope to return to this in a future work.

\subsection{Comparison with other work} As already mentioned, Lorentzian polynomials originate in \cite{BrandenHuh}, and have since been studied in a number of other works including \cite{Aluffi,branden2021lorentzian,dey2022polynomials,zbMATH07542930,BrandanHuhSpaces,MezarosSetiabrata,Eur}.  

The work we consider here is influenced to a great extent by work of Ross-Toma \cite{RT19,RT20} who gave an extension of the Hodge-Riemann bilinear relations that apply to Schur classes of ample vector bundles, and also to Schur polynomials in K\"ahler classes \cite{RT22}.    Since Schur polynomials are dually Lorentzian, we recover here many of the results of Ross-Toma for vector bundles that split as a sum of ample line bundles, and in fact extend them far beyond Schur polynomials (see Remarks \ref{rmk:RT1}-\ref{rmk:RT3} for more details).    A version of Theorem \ref{thm:introcombinatorics} for Schur polynomials appears in \cite{RT19} as a corollary of the geometric techniques used in that paper; here we have a proof that is purely algebraic.

As discussed above, a summary of the main idea of this paper is that any theory that admits a mixed AF inequality must also admit generalized versions coming from dually Lorentzian polynomials. There are likely a large number of other examples that we could include, for instance there should be a generalized AF inequality for nef classes on K\"ahler manifolds with certain numerical dimension that admit the Hodge-Riemann property as in \cite{hu2022hard}.     As this work was being completed we became aware of \cite{hu2023intersection} that has a similar theme.

\subsection{Organization} Section \S\ref{sec:preliminaries} has preliminaries including a convenient definition of Lorentzian polynomials, and in \S\ref{sec:preserving} we collect some simple linear operators that preserve the Lorentzian property.   Dually Lorentzian polynomials are defined in \S\ref{sec:duallyLorentzian} and there we prove the first half of Theorem \ref{thm:intropreserve} (namely that $\partial_s$ preserves the Lorentzian property).   Since we consider the notion of being dually Lorentzian useful in its own right, in \S\ref{sec:oper-pres-dual} we study some of its basic properties including giving various operators that preserve being dually Lorentzian.   In \S\ref{sec:strictly} we consider strictly Lorentzian polynomials, and prove the second half of Theorem \ref{thm:intropreserve} (namely that $\partial_s$ preserves the property of being strictly Lorentzian).  In \S\ref{sec:CLorentzian} we give the extension to $C$-Lorentzian polynomials, and finally in \S\ref{sec:applications} we collect the aforementioned applications.

\subsection{Acknowledgements}

This material is based upon work supported by the National Science Foundation under Grant No. DMS-1749447. The second named author was supported by the EPSRC grants EP/V013270/1, EP/V055445/1 and by funding from the Carl Zeiss Foundation.
The third named author was supported by DFG grant WA3510/3-1.
\section{Preliminaries}\label{sec:preliminaries}
In this paper we consider $\NN$ to be the non-negative integers and $[m]=\{0,\ldots, m\}$.
For an element $\alpha = (\alpha_1,\ldots,\alpha_n) \in  \NN^n$ we set $|\alpha|=\sum \alpha_i$. We denote the $i$-th canonical basis vectors in $\NN^^n$ by $e_i$ and the element $(d,\ldots,d) \in \NN^n$ by $\underline{\mathbf{d}}$. For two elements $\kappa, \gamma \in \NN^n$ we write $\gamma \leq \kappa$ if $\gamma_i \leq \kappa_i$ for $i=1,\ldots,n$.
We write $x^\alpha$ for the monomial in $\RR[x_1,\ldots,x_n]$ with exponent vector $\alpha$, i.e. $x^\alpha = \prod_{i=1}^n x_i^{\alpha_i}$ and we will denote by $\RR_\kappa[x_1,\ldots,x_n]$ the set of polynomials, which are supported in monomials $x^\alpha$ with $\alpha \leq \kappa$. In other words, $\RR_\kappa[x_1,\ldots,x_n]$ is the set of polynomials with multidegree $\kappa$, i.e. $x_i$-degree has degree at most $\kappa_i$ for all $i=1,\ldots, n$.  

We recall the definition of Lorentzian polynomials from \cite{BrandenHuh} or more precisely an equivalent characterisation of these polynomials.
\begin{definition}\
\begin{enumerate}
\item  A subset $C \subset \NN^n$ is called \emph{M-convex}, if for each pair $\alpha,\beta \in C$ with $\alpha_i > \beta_i$ there exists an index $j \in  \{1, \ldots,n\}$ such that $\alpha_j < \beta_j$ and $\alpha - e_i + e_j \in C$.
  
\item  A homogeneous polynomial $f \in \RR[x_1, \ldots x_n]$ of degree $d \geq 2$ with non-negative coefficients is called \emph{Lorentzian} if its support is M-convex and $\partial^\alpha f$ defines a quadratic form with at most one positive eigenvalue (counted with multiplicities) for every monomial differential operator
  $\partial^\alpha \in \RR[\partial_{x_1}, \ldots, \partial_{x_n}]$ of degree $|\alpha|=d-2$.

\item  A Lorentzian polynomial of degree $d$ is called \emph{strictly Lorentzian} if it is supported in all monomials of degree $d$ and the quadratic form $\partial^\alpha f$ has signature $(+,-,\ldots,-)$ for every $\alpha \in \NN^n$ with $|\alpha|=d-2$.

\item  We extend this definition to degree $0$ and $1$ by saying that all non-negative constants and linear polynomials with non-negative coefficients are Lorentzian and all positive constants and all linear forms with positive coefficients at every variable are strictly Lorentzian.
 
\item  We denote the set Lorentzian polynomials of degree $d$ by $\mathrm L^d_n$ and the subset of strictly Lorentzian polynomials by $\sL^d_n$.
\end{enumerate}
\end{definition}

The reader will easily check that any monomial is Lorentzian.  The following two criteria are useful to check the Lorentzian property.
\begin{lemma}
  \label{lem:sylvester-criterion}
  Consider a symmetric $m\times m$ matrix $M$ with non-negative entries. Then the following two conditions are equivalent.
  \begin{enumerate}
  \item The matrix $M$ has at most one positive eigenvalue. \label{item:lorentzian-prop}
  \item For every $I\subset \{1,\ldots,m\}$ we have $(-1)^{|I|} \det M_I \leq 0$  where $M_I$ denotes the principal minor of $M$ associated to $I$. \label{item:minor-sign}
  \end{enumerate}
\end{lemma}
\begin{proof}
Our proof follows \cite[Lemma 2.3]{vH21}.  Suppose (1) holds.  On the one hand, by Cauchy's interlacing theorem $M_I$ has at most one positive eigenvalue. On the other hand, since $M_I$ has non-negative entries, the Perron-Frobenius theorem implies the existence of a non-negative eigenvalue. This shows that $(-1)^{|I|} \det M_I \leq 0$.

 The other direction is proved by induction on $m$,  the statement being trivial for $m\le 2$.    Assuming then the result for $m-1$ assume for contradiction that $v,w$ are linearly independent eigenvectors of $M$ with positive eigenvalue, which we may assume to orthogonal with respect to $M$.     As $(-1)^m \det(M)\le 0$ there must be a third eigenvector $u$ orthogonal to $v$ and $w$ with non-negative eigenvalue.

Choose any $i$ with $u_i\neq 0$ and let $I = [m]\setminus \{i\}$, and then choose $a,b\in \mathbb R$ so that $x = v-au$ and $y = w-bu$ satisfy $x_i=y_i=0$.  Then by construction $x,y$ are linearly independent and $z^TMz>0$ for all non-zero $z$ in the span of $x,y$.  As $x$ and $y$ are supported on $I$ this implies that $M_I$ has at least two positive eigenvalues which contradicts the inductive hypothesis.
\end{proof}

 Variants of the following lemma appear in \cite[Lemma 2.5]{BrandenHuh} and \cite[Lemma 2.1]{vH21}. For $x\in \RR^n$ we write $x\geq 0$ if $x_i\geq 0$ for all $i$. 
	
 \begin{lemma}\label{lemma:hodge}
	Let $q\colon \RR^n\to \RR$ be a quadratic form and suppose that $q(y)>0$ if  $y> 0$. Then the following statements are equivalent:
	\begin{enumerate}
		\item If $y>0$, then $q(x,y)^2 \geq q(x)q(y)$ for all $x\geq 0$ and equality holds if and only $x$ is proportional to $y$.
		\item If $y>0$, then $q(x,y)^2 \geq q(x)q(y)$ for all $x$ and equality holds if and only if $x$ is proportional to $y$. 
		\item If $q(y)>0$, then $q(x,y)^2 \geq q(x)q(y)$ for all $x$ and equality holds if and only $x$ is proportional to $y$. 
		\item $q$ has signature $(+,-,\ldots, -)$.
	\end{enumerate}
\end{lemma}
\begin{proof}
	(1)$\Rightarrow$(2): Fix $x\in \RR^n$ and choose a number $t$ so that $x+ty>0$. By (1) we have 
	$$q(x+ty,y)^2 \geq q(x+ty,x+ty)q(y).$$ 
	Expanding both sides yields (2), since the terms involving $t$ cancel.
	
	(2)$\Rightarrow$(4): Fix $y>0$. Then $q(y)>0$ by assumption. If $x$ is orthogonal to $y$ with respect to $q$, then $q(x)\leq 0$ with equality if and only if $x=0$. Hence $q$ has signature  $(+,-,\ldots, -)$.
	
	(4)$\Rightarrow$(3): Suppose $q(y)>0$ and let $x$ be arbitrary. Put $z= x-ty$ with $t= q(x,y)/q(y)$. Then 
	$q(z,x)=0$ and hence $q(z)\leq 0$ with equality if and only if $z=0$. Expressing this in $x$ and $y$ yields (3). 
	
	(3)$\Rightarrow$(1) is trivial. 
\end{proof}

\section{Truncations}\label{sec:preserving}
Operators preserving the Lorentzian property play a central role in the theory developed by  Br\"and\'en-Huh, and we will need the following aspect:

\begin{theorem}\label{thm:changeVar}
	Let $f\in \mathrm L_n^d$ and $g\in \mathrm L_n^e$.
	\begin{enumerate}
		\item For every $m\times n$ matrix $A$ with non-negative entries $f(Ax)\in \mathrm L^d_m$.  
		\item $fg\in \mathrm L_n^{d+e}$. 
	\end{enumerate}
\end{theorem}
\begin{proof}
 Item (1) is  \cite[Theorem~2.10]{BrandenHuh}.  Item (2) is an easy consequence of (1) upon considering the polynomial $f(x)g(y)$ restricted to $x=y$.
\end{proof}

The purpose of this section is to show that truncations  preserve the Lorentzian property.

\begin{definition}
  Let $f = \sum_{\alpha \in \NN^n}\lambda_\alpha x^\alpha$ be homogeneous polynomial of degree $d$. Then its $\kappa$-\emph{truncations} are defined as
  \[f_ {\leqslant \kappa} = \sum_{\alpha \leq \kappa} \lambda_\alpha x^\alpha \qquad \text{ and }\qquad  f_ {\geqslant \kappa} = \sum_{\alpha \geq \kappa} \lambda_\alpha x^\alpha.\]
\end{definition}

\begin{proposition}
  \label{prop:truncation}
  If $f \in \RR[x_1,\ldots,x_n]$ is Lorentzian, so are its truncations.
\end{proposition}
\begin{proof}
  The claim follows by induction from the following  lemmas.
\end{proof}

\begin{lemma}
Fix an index $1 \leq i \leq n$
  Given a polynomial $f=\sum_\alpha \lambda_\alpha x^\alpha$ with M-convex support. Then $\bar f = \sum_{\alpha, \alpha_i \leq m}\lambda_\alpha x^\alpha$ has M-convex support, as well. 
\end{lemma}
\begin{proof}
  Assume we have $\alpha, \beta \in \supp(\bar f)$ with $\alpha_j > \beta_j$. Then by the M-convexity for $f$ there exists an index $k$ with $\alpha_k < \beta_k$ and $\alpha -e_j +e_k \in \supp(f)$. If $k \neq i$, then also $\alpha-e_j + e_k \in \supp(\bar f)$ holds. In the case $k=i$,  we have $\alpha_k < \beta_k \leq m$, since $\beta \in \supp(\bar f)$.  Hence, we obtain
  \[\alpha_k -e_j +e_k \in \supp(f) \cap \{\alpha \mid \alpha_i \leq m\} = \supp(\bar f).\]
\end{proof}

\begin{lemma}
Fix an index $1 \leq i \leq n$
  Given a polynomial $f=\sum_\alpha \lambda_\alpha x^\alpha$ with M-convex support. Then $\bar f = \sum_{\alpha, \alpha_i \geq m}\lambda_\alpha x^\alpha$ has M-convex support, as well. 
\end{lemma}
\begin{proof}
  Assume we have $\alpha, \beta \in \supp(\bar f)$ with $\alpha_j > \beta_j$. Then by the M-convexity for $f$ there exists an index $k$ with $\alpha_k < \beta_k$ and $\alpha -e_j +e_k \in \supp(f)$. If $j \neq i$, then also $\alpha-e_j + e_k \in \supp(\bar f)$ holds. In the case $j=i$,  we have $\alpha_j > \beta_j \geq m$, since $\beta \in \supp(\bar f)$.  Hence, we obtain
  \[\alpha_k -e_j +e_k \in \supp(f) \cap \{\alpha \mid \alpha_i \geq m\} = \supp(\bar f).\]
\end{proof}

\begin{lemma}
  Fix an index $1 \leq i \leq n$
  Given a Lorentzian polynomial $f=\sum_\alpha \lambda_\alpha x^\alpha$ of degree $d$ we set  $\bar f = \sum_{\alpha, \alpha_i \leq m}\lambda_\alpha x^\alpha$. Then for every monomial $\partial^\alpha \in \RR[\partial_{x_1},\ldots,\partial_{x_n}]$ of degree $d-2$ the quadratic form $\partial^\alpha \bar f$ has at most one positive eigenvalue.
\end{lemma}
\begin{proof}
  Given a monomial differential operator $\partial^\alpha$ of degree $d-2$, we consider 
  the coefficients $c_{\beta}$ of $\partial^\alpha f$. We then set \[c_{ij}=c_{ji}=
    \begin{cases}
      \frac{1}{2}c_{e_i+e_j} & i \neq j\\
      c_{2e_i}                & i = j
    \end{cases}
    .\]
  Because $f$ was assumed to be Lorentzian, by Lemma~\ref{lem:sylvester-criterion} we have $(-1)^{|I|}|M_I| \leq 0$ for the principal minors of the matrix $M=(c_{ij})$. Here, $I \subset \{1,\ldots,n\}$ denotes the subset of those indices, which are included in the principal minor.  We will prove that the same condition on the minors is fulfilled for the matrix $\bar M$ corresponding to $\partial^\alpha  \bar f$. Note, that $\bar M$ coincides with $M$ except that some of the entries might have been changed to $0$.

  Let us first note the following facts:
  \begin{enumerate}
  \item If $\alpha_i > m$, then $\partial^\alpha \bar f=0$. 
  \item If $\alpha_i \leq m-2$, then  $\partial^\alpha \bar f= \partial^\alpha f$.
  \end{enumerate}
  In particular, in each of these cases $\partial^\alpha \bar f$ has at most one positive eigenvalue.
  Thus it remains to consider the cases $\alpha_i=m$ and $\alpha_i=m-1$.

We start with $\alpha_i=m$. Then $\partial^\alpha f-\partial^\alpha \bar f$ is supported
  in exactly those monomials of $\partial^\alpha f$ which contain $x_i$. Hence, by the Eigenvalue Interlacing Theorem, 
  $\partial^\alpha \bar f$ has again at most one positive eigenvalue.

  Next, we consider the case $\alpha_i=m-1$. Here, $\partial^\alpha f-\partial^\alpha \bar f$ is supported in the monomial $x^{2}_i$ with coefficient $c_{ii}$. Hence, $\bar M$ arises from $M$ by setting the $i$-th diagonal entry to $0$. 
  
  If $i \notin I$, then $M_I=\bar M_I$ and $(-1)^{|I|}|\bar M_I| = (-1)^{|I|}|M _I| \le 0$.   If $i \in I$, then we have $|M_I| = |\bar M_I| + c_{ii}|M_{I\setminus \{i\}}|$, by Laplace expansion  along the $i$-th row for both matrices.  So
  \begin{align*}
    (-1)^{|I|}|\bar M_I|&= (-1)^{|I|}\left(|M_I| \;-\; c_{ii}|M_{I\setminus \{i\}}|\right)\\
                    &=  (-1)^{|I|}|M_I| \;+\; (-1)^{|I\setminus\{i\}|}\cdot c_{ii} \cdot |M_{I\setminus \{i\}}|\\
    &\leq 0.
  \end{align*}
where the last inequality uses Lemma~\ref{lem:sylvester-criterion} applied to $M$ and $c_{ii} \geq 0$ since $f$ is Lorentzian.  Thus the result we want follows by Lemma~\ref{lem:sylvester-criterion}.
\end{proof}

\begin{lemma}
  Fix an index $1 \leq i \leq n$
  Given a Lorentzian polynomial $f=\sum_\alpha \lambda_\alpha x^\alpha$ of degree $d$ we set  $\bar f = \sum_{\alpha, \alpha_i \geq m}\lambda_\alpha x^\alpha$. Then for every monomial $\partial^\alpha \in \RR[\partial_{x_1},\ldots,\partial_{x_n}]$ of degree $d-2$ the quadratic form $\partial^\alpha \bar f$ has at most one positive eigenvalue.
\end{lemma}
\begin{proof}
  Given a monomial differential operator $\partial^\alpha$ of degree $d-2$, we as before consider 
  the coefficients $\bar c_{\beta}$ of $\partial^\alpha \bar f$ and set \[\bar c_{ij}= \bar c_{ji}=
    \begin{cases}
      \frac{1}{2}\bar c_{e_i+e_j} & i \neq j\\
      \bar c_{2e_i}                & i = j
    \end{cases}
    .\]
  We set $\bar M = (\bar c_{ij})$.
  
  Similar to the case for the upper truncation observe that
  \begin{enumerate}
  \item If $\alpha_i \geq m$, then $\partial^\alpha \bar f=\partial^\alpha  f$. 
  \item If $\alpha_i < m-2$, then  $\partial^\alpha \bar f= 0$.
  \end{enumerate}
and so in each case $\partial^\alpha \bar f$ has at most one positive eigenvalue.  Thus it remains to consider the cases $\alpha_i=m-2$ and $\alpha_i=m-1$.

If $\alpha_i=m-2$, then $\partial^\alpha \bar f$ is supported only in $x_i^2$. Hence, the quadratic form has rank $1$ and therefore exactly one non-zero eigenvalue.

It remains to consider the case $\alpha_i=m-1$. Here, $\partial^{\alpha} f-\partial^\alpha \bar f$ is supported in the monomials which contain  $x_i$.  The matrix $\bar M$ has therefore non-zero entries only in the $i$-th row and column. The only non-zero principal minors have the form
  \[
    \left|
      \begin{matrix}
        c_i &c_{ij}\\
        c_{ij} &0
      \end{matrix}
\right| = -c_{ij}^2 \leq 0.
\]
Thus the result we want follows by Lemma~\ref{lem:sylvester-criterion}.
\end{proof}

\section{Dually Lorentzian polynomials}\label{sec:duallyLorentzian}

We introduce in this section the main concept of this paper, namely dually Lorentzian polynomials. In the main result of this section (Theorem~\ref{thm:operator-criterion}) we characterize dually Lorentzian polynomials as precisely those constant coefficients differential operators that preserve the Lorentzian property. We then derive some easy consequences of this statement and discuss an interesting class of examples.  

\begin{definition}[Normalization]
  The \emph{Normalization} operator $N$ on polynomials is the linear operator satisfying $N(x^\alpha)= \frac{1}{\alpha!}x^\alpha$.
  \end{definition}
  
  \begin{definition}[Dually Lorentzian Polynomials]
  For an element $s \in \RR_{\kappa}[x_1, \ldots x_n]$ we set $${s}^\vee := N(x^{\kappa}\cdot s(x^{-1}_1, \ldots, x^{-1}_n)).$$ 
  We say a homogeneous polynomial $s$ is \emph{dually Lorentzian} if $s^\vee$ is Lorentzian.
  \end{definition}

\begin{remark}
  For a given $s$ the dual $s^\vee$ is not uniquely defined since $\kappa$ is not fixed. However, by \cite[Lem.~7]{zbMATH07542930} the dual Lorentzian property is independent of this choice.
\end{remark}

\begin{theorem}
  \label{thm:operator-criterion}
  Let $s \in \RR_{\kappa}[x_1, \ldots x_n]$ be a homogeneous polynomial. Then the corresponding differential operator $\partial_s:=s(\partial_{x_1},\ldots,\partial_{x_n}) \in \RR[\partial_{x_1},\ldots,\partial_{x_n}]$ preserves the Lorentzian property if and only if $s$ is dually Lorentzian.
\end{theorem}
\begin{proof}    Let $\partial^\alpha \in \RR[\partial_{x_1},\ldots,\partial_{x_n}]$ be a monomial differential operator, and consider any $\gamma \geq \kappa$.
Then the symbol $\operatorname{sym}_{\partial^\alpha}$ of the linear operator
  \[
    \RR_{\gamma}[x_1,\ldots,x_n] \to \RR_{\gamma}[x_1,\ldots,x_n], f \mapsto \partial^\alpha f
  \]
  in the sense of \cite[Thm 3.2]{BrandenHuh} is
  \[\operatorname{sym}_{\partial^\alpha}(x,w)=\sum_{\beta \leq \gamma} {\gamma \choose \beta} \partial^\alpha(x^\beta)w^{\gamma-\beta} = \partial^\alpha\!\left(\prod_i(x_i+w_i)^{\gamma_i}\right).\]
Hence, we obtain
  \[\operatorname{sym}_{\partial^\alpha} = \frac{\gamma!}{(\gamma-\alpha)!}\prod_i(x_i+w_i)^{\gamma_i-\alpha_i}= \frac{\gamma!}{(\gamma-\alpha)!} \prod_i (x_i+w_i)^{\gamma_i}(x_i+w_i)^{-\alpha_i}.\]
  Thus, by linearity the symbol of $\partial_s \colon  \RR_{\gamma}[x_1,\ldots,x_n] \to \RR_{\gamma}[x_1,\ldots,x_n]$ with $s=\sum_\alpha \lambda_\alpha x^\alpha$ is given by
  \[
    \operatorname{sym}_{\partial_s}=\sum_\alpha \lambda_\alpha \cdot \frac{\gamma!}{(\gamma-\alpha)!} \cdot \prod_i (x_i+w_i)^{\gamma_i} \cdot \prod_i (x_i+w_i)^{-\alpha_i}.\]
  Now, by Theorem \ref{thm:changeVar}(1) the polynomial $\operatorname{sym}_{\partial_s}$ is Lorentzian if and only if the same is true for
  \begin{align*}
  \sum_\alpha \lambda_\alpha \frac{\gamma!}{(\gamma-\alpha)!} \cdot x^{\gamma} \cdot  x^{-\alpha} &= \gamma!  \cdot x^{\gamma} \cdot \sum_\alpha \frac{\lambda_\alpha}{(\gamma-\alpha)!} \cdot   x^{-\alpha}\\
                                   &= \gamma!\cdot N(x^{\gamma}\cdot s(x_1^{-1},\ldots,x^{-1}_n)),
  \end{align*}
 but this is precisely what we are assuming as $s$ is dually Lorentzian.   Thus, by \cite[Thm 3.2.]{BrandenHuh}, the operator $\partial_s \colon  \RR_{\gamma}[x_1,\ldots,x_n] \to \RR_{\gamma}[x_1,\ldots,x_n]$ maps Lorentzian polynomials to Lorentzian polynomials for arbitrary $\gamma \geq \kappa$. Hence,  $\partial_s$ preserves the Lorentzian property for arbitrary polynomials in $\RR[x_1,\ldots,x_n]$.

  For the other direction assume that $\partial_s$ maps Lorentzian polynomials to Lorentzian polynomials. Since every monomial is Lorentzian,  by assumption we obtain the Lorentzian property also for \[\partial_s(x^{\kappa}) =  \kappa!\cdot N(x^{\kappa}\cdot s(x_1^{-1},\ldots,x^{-1}_n)).\]
\end{proof}

We now turn to some examples of dually Lorentzian polynomials.

\begin{theorem}
  \label{thm:schubert}
  Schubert polynomials  (and therefore Schur polynomials) are dually Lorentzian. 
\end{theorem}
\begin{proof}
This is proved in {\cite[Thm. 6]{zbMATH07542930}}.  In essence, the authors there show that for every Schubert polynomial $\mathfrak{S}_w$ the polynomial $\mathfrak{S}_w^\vee$ agrees with the volume polynomial of a collection of nef line bundles on an irreducible variety in a product of projective spaces that is constructed as a degeneracy locus. Hence, by \cite[Theorem 4.6.]{BrandenHuh} $\mathfrak{S}_w^\vee$ is Lorentzian and therefore $\mathfrak{S}_w$ is dually Lorentzian. Furthermore, for the special case of Schur polynomials the authors use \ref{lem:Schur} below to show that not only Schur polynomials are dually Lorentzian, but also that  normalized Schur polynomials $N(s_\lambda)$ are Lorentzian, see the proof of {\cite[Thm. 3]{zbMATH07542930}}.
\end{proof}

	\begin{lemma}
		\label{lem:Schur} For every partition $\lambda$ of length $\leq n$ and degree $|\lambda|= d$, 
		$$x_{1}^d\cdots x_n^d\cdot s_\lambda(x_1^{-1}, \ldots, x_n^{-1}) = s_\kappa(x_1,\ldots,x_n)$$ where 
		$\kappa=(d-\lambda_n,\ldots,d-\lambda_1)$.
	\end{lemma}
\begin{proof}
	The character of every rational  representation $W$  of $\GL_n(K)$ satisfies $$\chi_{W*}(x_1,\ldots,x_n)=\chi_{W}(x_1^{-1},\ldots,x_n^{-1})$$ and the character of 
	$\det \colon \GL_n(K)\to \GL_1(K)$, is $x_{1}\cdots x_n$. Hence if $W$ is the $\GL_n(K)$-representation with highest weight $\lambda$, 
	$$ x_{1}^d\cdots x_n^d\cdot s_\lambda(x_1^{-1}, \ldots, x_n^{-1}) = \chi_{\det^d \otimes W^*}(x_1,\ldots, x_n).$$
	The highest weight of $\det^d \otimes W^*$ is  $\kappa=(d-\lambda_n,\ldots,d-\lambda_1)$. Therefore $\det^d \otimes W^*$ is polynomial and its character is $s_\kappa(x_1,\ldots, x_n)$. 
\end{proof}

\begin{example} From the above discussion for any Schur polynomial $s_{\lambda}$ both $N({s_\lambda})$ and  $s^\vee_{\lambda}$ are Lorentzian.   On the other hand, the polynomial $f=w^2+3wx+3x^2+3wy+3xy+2y^2+3wz+2xz+2yz+z^2$ is dually Lorentzian, but $N(f)$ is not Lorentzian. \end{example}

\begin{example}\label{ex:duallylorentzianbutnotschurpositive}
Consider the difference of Schur polynomials
\begin{align*}
p(x_1,x_2,x_3)&= s_{2,1}(x_1,x_2,x_3) - s_{1,1,1}(x_1,x_2,x_3) \\
&= x_0^2x_1 + x_0x_1^2 + x_0^2x_2 + x_0x_1x_2 + x_1^2x_2 + x_0x_2^2 + x_1x_2^2.
\end{align*}
A direct computation shows that $N(p)$ is Lorentzian.  Thus by Lemma \ref{lem:Schur}  the polynomial $$q(x_1,x_2,x_3):=s_{3,2,1}(x_1,x_2,x_3) - s_{2,2,2}(x_1,x_2,x_3)$$ is dually Lorentzian but clearly not Schur positive
\end{example}

\begin{proposition}
  \label{prop:dual-product}
  If $f$ and $g$ are dually Lorentzian, so is their product $fg$.
\end{proposition}
\begin{proof}
  This is an immediate consequence of Theorem~\ref{thm:operator-criterion}, as $\partial_{fg} = \partial_f \circ \partial_g$.
\end{proof}

\begin{corollary}
\label{cor:dual-truncation}
  If $f$ is dually Lorentzian, so are its truncations.
\end{corollary}
\begin{proof}
  We have 
  \begin{equation}\label{eq:trunc}(f_{\geqslant \gamma})^\vee= (f^\vee)_{\leqslant \kappa-\gamma}\quad \text{and}\quad (f_{\leqslant \gamma})^\vee= (f^\vee)_{\geqslant \kappa-\gamma}.\end{equation}
   Hence the claim follows from Proposition~\ref{prop:truncation}.
\end{proof}

\section{Operators preserving the dual Lorentzian property}
\label{sec:oper-pres-dual}
Note, that by definition the map $\dottedbox^\vee \colon f \mapsto f^\vee$ is a bijective linear operator between dually Lorentzian polynomials and Lorentzian polynomials in $\RR_{\kappa}[x_1, \ldots,x_n]$. This has inverse
\[
  \dottedbox^\wedge \colon f \mapsto f^\wedge := x^\kappa \cdot (N^{-1}(f))(x_1^{-1},\ldots,x_n^{-1}),
\]
where $N^{-1}(\sum_{\alpha} c_\alpha x^{\alpha}) := \sum \alpha! c_{\alpha} x^{\alpha}$.

Let \[T \colon \RR_\kappa[x_1,\ldots,x_n] \to \RR_\gamma[x_1,\ldots,x_n]\] be a homogeneous linear operator. We  define its dual by
\[T^\vee \colon \RR_\kappa[x_1,\ldots,x_n] \to \RR_\gamma[x_1,\ldots,x_n];\quad f \mapsto T(f^\wedge)^\vee,\]
resulting in the commutative diagram
\[  \begin{tikzcd}
    \RR_{\kappa}[x_1, \ldots,x_n] \arrow[r,"T"]\arrow[d,"\dottedbox^\vee"] & \RR_{\gamma}[x_1, \ldots,x_n]\\
      \RR_{\kappa}[x_1, \ldots,x_n] \arrow[r,"T^\vee"]  &\RR_{\gamma}[x_1, \ldots,x_n]\arrow[u,"\dottedbox^\wedge"] 
    \end{tikzcd}\]
We define the \emph{co-symbol} of $T$ to be the symbol $\operatorname{sym}_{T^\vee}$ of $T^\vee$.

\begin{lemma}
\label{lem:dual-symbol}
  If the co-symbol $\operatorname{sym}_{T^\vee}$ is Lorentzian, then the operator $T$ preserves the dual Lorentzian property.
\end{lemma}
\begin{proof}
  It follows from the observation above, that $T$ preserves the dual Lorentzian property if and only if $T^\vee$ preserves the Lorentzian property.
  By \cite[Thm 3.2]{BrandenHuh} the operator $T^\vee$ preserves the Lorentzian property if $\operatorname{sym}_{T^\vee}$ is Lorentzian.
\end{proof}

\begin{example}
  Consider the linear operator $M \colon \RR_\kappa[x_1,\ldots,x_n] \to \RR_\kappa[x_1,\ldots,x_n]$ defined by $M(x^\beta)=\frac{1}{(\kappa-\beta)!}x^\beta$. Then it is easy to see that $M^\vee=N$.  Since by \cite[Cor.~3.7]{BrandenHuh} $N$ preserves the Lorentzian property, we deduce that $M$ preserves the dual Lorentzian property.
\end{example}

 


\begin{lemma}
  \label{lem:T-dual-derivative1}
    Consider the linear differential operator $\partial^\alpha$ on $\RR_\kappa[x_1,\ldots,x_n]$.  Then we have 
    \[(\partial^{\alpha})^\vee(x^\beta)=
    \begin{cases}
      \frac{(\kappa-\beta)!\beta!}{(\kappa-\beta-\alpha)!(\beta+\alpha)!}\cdot x^{\beta+\alpha} & \alpha+\beta \leq \kappa\\
      0 & \text{otherwise.}      
    \end{cases}
  \]
\end{lemma}
\begin{proof}
  By elementary algebra and \eqref{eq:trunc} we obtain
    \begin{align*}
    \partial^\alpha((x^\beta)^\wedge)^\vee = \partial^\alpha\left(\beta!\cdot x^{\kappa-\beta}\right)^\vee
                                    &= \left(\left( \frac{(\kappa-\beta)!\beta!}{(\kappa-\beta-\alpha)!}\cdot x^{\kappa-\beta-\alpha}\right)_{\geq 0}\right)^\vee\\
                                     &= \left(\frac{(\kappa-\beta)!\beta!}{(\kappa-\beta-\alpha)!(\beta+\alpha)!}\cdot x^{ \beta+\alpha}\right)_{\leq \kappa}.
  \end{align*}
\end{proof}
We now prove a partial analogue of Theorem~\ref{thm:operator-criterion} for the dual Lorentzian property.
\begin{proposition}
  \label{prop:operator-criterion-dual}
  Let $f \in \RR[x_1,\ldots,x_n]$ be a polynomial. The differential operator $\partial_f =f(\partial_{x_1},\ldots,\partial_{x_n}) \in \RR[\partial_{x_1},\ldots,\partial_{x_n}]$ preserves the dual Lorentzian property if $f$ is Lorentzian.
\end{proposition}
\begin{proof}
  By Lemma~\ref{lem:dual-symbol}  it is sufficient to show that the co-symbol
  $\operatorname{sym}_{\partial^\vee}$ is Lorentzian. Now, by applying Lemma~\ref{lem:T-dual-derivative1} and taking into account linearity, for the co-symbol we obtain
    \begin{align*}
      \operatorname{sym}_{(\partial^\alpha)^\vee}(x,u) &=  \left(\sum_{\beta\le \kappa} \binom{\kappa}{\beta} \frac{(\kappa-\beta)!\beta!}{(\kappa-(\beta+\alpha))!(\beta+\alpha)!}\cdot x^{\beta+\alpha} u^{\kappa-\beta}\right)_{\leq(\kappa,\kappa)}\\
 &= \left(\sum_{\beta\le \kappa} \binom{\kappa}{\beta+\alpha}  x^{\beta+\alpha} u^{\kappa-\beta}\right)_{\leq(\kappa,\kappa)}\\
               &= \left( \sum_{\alpha \leq \gamma \leq \kappa} {\kappa \choose \gamma} x^{\gamma}u^{\kappa-\gamma+\alpha}\right)_{\leq(\kappa,\kappa)}\\
               &= \left(\sum_{\gamma \leq \kappa} {\kappa \choose \gamma}  x^{\gamma}u^{\kappa-\gamma+\alpha}\right)_{\leq(\kappa,\kappa)}\\
               &= \left(u^\alpha \sum_{\gamma \leq \kappa} {\kappa \choose \gamma}  x^{\gamma}u^{\kappa-\gamma}\right)_{\leq(\kappa,\kappa)}\\
             &=  \left(u^\alpha (x +u)^{\kappa} \right)_{\leq(\kappa,\kappa)}    
    \end{align*}
    By linearity for the co-symbol of $\partial_f$ we obtain
    \[\operatorname{sym}_{\partial_f^\vee}= \left(f(u_1,\ldots,u_n)(x+u)^\kappa \right)_{\leq(\kappa,\kappa)}.\]
    But $f(u_1,\ldots,u_n)$ was Lorentzian by assumption. Therefore, $\operatorname{sym}_{\partial_f^\vee}$ is Lorentzian by
    Theorem~\ref{thm:changeVar} and Proposition~\ref{prop:truncation}. Now, the claim follows by Lemma~\ref{lem:dual-symbol}.
  \end{proof}

  \begin{remark}
  In general, the co-symbol $\operatorname{sym}_{T^\vee}$   differs from the dual  $\left(\operatorname{sym}_T\right)^\vee$ of the symbol of $T$. Indeed,
  for $\partial^\alpha$ the co-symbol was given by
  \[ \left(\sum_{\beta\le \kappa} \binom{\kappa}{\beta} \frac{(\kappa-\beta)!\beta!}{(\kappa-\beta-\alpha)!(\beta+\alpha)!}\cdot x^{\beta+\alpha} u^{\kappa-\beta}\right)_{\leq(\kappa,\kappa)}\]
  whereas the dual of the symbols is equal to
   \[\sum_{\alpha \leq \beta\le \kappa} \binom{\kappa}{\beta} \frac{\beta!}{(\beta-\alpha)!(\kappa+\alpha)!}  x^{\kappa-\beta+\alpha} u^{\beta} = \sum_{\alpha \leq \beta\le \kappa} \binom{\kappa}{\beta} \frac{(\kappa-\beta)!}{(\kappa-\beta-\alpha)!(\kappa+\alpha)!}  x^{\beta+\alpha} u^{\kappa-\beta}.\]

\end{remark}

\begin{proposition}
  \label{prop:dual-and-derivative}
  Let $s\in \RR[x_1,\ldots,x_n]$ be dually Lorentzian and $\partial = \sum_i a_i \partial_{x_i}\in \RR[\partial_{x_1},\ldots,\partial_{x_n}]$ a linear differential operator with non-negative constant coefficients. Then $\partial s$ is again dually Lorentzian.
\end{proposition}
\begin{proof}
  By arguing as in the proof of Proposition~\ref{prop:operator-criterion-dual} while setting $f=\sum_i a_i x_i$ we obtain \[\operatorname{sym}_{\partial^\vee}(x,u)=\left(\sum_i a_i u_i (x +u)^{\kappa} \right)_{\leq(\kappa,\kappa)}.\]

Theorem~\ref{thm:changeVar} and  Proposition~\ref{prop:truncation} imply that this polynomial is Lorentzian. 
 
\end{proof}

\begin{corollary}
  \label{cor:derived}
If $s \in \RR[x_1,\ldots,x_n]$ is dually Lorentzian then the \emph{derived polynomial} $$s^{(1)} := \sum_i \partial_{x_i} s$$ is dually Lorentzian.
\end{corollary}

Repeated application of Proposition~\ref{prop:dual-and-derivative} yields
\begin{corollary}
  If $s\in \RR[x_1,\ldots,x_n]$ is dually Lorentzian then $\partial^\alpha s$ is also dually Lorentzian.
\end{corollary}

In the following we consider the formal antiderivative defined on the monomials by
$\int x^\beta dx^\alpha= \frac{\beta!}{(\beta+\alpha)!}x^{\beta+\alpha}$.

\begin{corollary}
  If $f$ is (dually) Lorentzian, then  the antiderivatives $\sum_i a_i \int f dx_i$ and $\int f dx^\alpha$ are again (dually) Lorentzian for any choices of $a_1,\ldots,a_n \geq 0$ and $\alpha \in \NN^n$.
\end{corollary}
\begin{proof}
  We start with the claim for the Lorentzian property. Set $d=\deg(f)-1$ and consider the operator $\partial = \sum_i a_i \partial_{x_i}$ on $\RR_{\underline{\mathbf{d}}}[x_1,\ldots,x_n]$.  Then we know from Proposition~\ref{prop:dual-and-derivative} that $\partial^\vee(f)$ is again Lorentzian, since $\partial^\vee$ preserves the Lorentzian property. This implies that also $\frac{1}{d}\partial^\vee(f)$ is Lorentzian. On the other hand, we have  \[\frac{1}{d}\partial^\vee(x^\beta)= \left(\sum_{i=1}^n a_i \frac{d-\beta_i}{d}\frac{1}{\beta_i+1}x^{\beta+e_i}\right)_{\leq \underline{\mathbf{d}}}.\]
  Note, that the definition of $\partial^\vee$ depends on the choice of $d$, but for $d' > d$ we still have $f \in  \RR_{\underline{\mathbf{d}}'}[x_1,\ldots,x_n]$. Now, for fixed $\beta$ and $d \to \infty$ we obtain
  \[\frac{1}{d}\partial^\vee(x^\beta) \to \sum_i a_i \frac{1}{\beta_i+1}x^{\beta+e_i}= \sum_i a_i\int x^\beta dx_i\]
  By linearity also $\frac{1}{d}\partial^\vee(f) \to \sum_i a_i \int f dx_i$ holds. Since the subset of Lorentzian polynomials is closed by \cite[Thm 2.25]{BrandenHuh}, this limit is necessarily Lorentzian. The claim for $\int f dx^\alpha$ follows by induction.

  The claim for the dual Lorentzian property follows in a similar manner, by using the fact that $\partial$ preserves the Lorentzian property by \cite[Cor.~2.11]{BrandenHuh}.
\end{proof}

\begin{lemma}
\label{lem:linear-trafo1}
Let $a \in \RR$ and
\[\partial \colon \RR_{\kappa}[x_1,\ldots,x_n] \to \RR_{\kappa+e_j}[x_1,\ldots,x_n]; \quad f \mapsto (1+ax_j\partial_{x_i}) f\]
Then in our notation from above we have 
$$\partial^\vee(x^\beta)= \left(x^\beta+ a \frac{\kappa_i-\beta_i}{\beta_i+1}x^{\beta+e_i}\right)_{\leq \kappa+e_j}.$$
\end{lemma}
\begin{proof}
  By elementary algebra we obtain
  \begin{align*}
    \partial((x^\beta)^\wedge)^\vee = \partial\left(\beta!\cdot x^{\kappa-\beta}\right)^\vee
                                    &= \left(\left(\beta!\cdot x^{\kappa-\beta}+ a\cdot \beta!\cdot (\kappa_i-\beta_i)x^{\kappa-\beta-e_i+e_j}\right)_{\geq 0}\right)^\vee\\
                                     &= \left(x^{\beta+e_j} + a\frac{\kappa_i-\beta_i}{\beta_i+1} x^{\beta+e_i}\right)_{\leq \kappa +e_j}.
  \end{align*}
\end{proof}

\begin{lemma}
  \label{lem:linear-trafo2}
  Let $a \in \RR_{\geq 0}$ and $i\neq j$.  Then $\partial=(1+ax_j\partial_{x_i})$ preserves the dual Lorentzian property.
\end{lemma}
\begin{proof}
  We argue as in Proposition~\ref{prop:dual-and-derivative}. We show that the co-symbol $\operatorname{sym}_{\partial^\vee}$ is Lorentzian and then apply Lemma~\ref{lem:dual-symbol} to obtain the result. Indeed, taking into account Lemma~\ref{lem:linear-trafo1} we conclude
  \begin{align*}
\operatorname{sym}_{\partial^\vee}(x,u) &= \left((x+u)^{\kappa+e_j} + \sum_{\beta\le \kappa} \binom{\kappa}{\beta} a\frac{\kappa_i-\beta_i}{\beta_i+1}   x^{\beta+e_i} u^{\kappa-\beta}\right)_{\leq(\kappa+e_j,\kappa+e_j)}\\
 &= \left((x+u)^{\kappa+e_j}  + a \sum_{\beta\le \kappa} \binom{\kappa}{\beta+e_i}  x^{\beta+e_i} u^{\kappa-\beta}\right)_{\leq(\kappa+e_j,\kappa+e_j)}\\
               &= \left((x+u)^{\kappa+e_j} + a \sum_{e_i \leq \beta \leq \kappa} {\kappa \choose \beta} x^{\beta}u^{\kappa-\beta+e_i}\right)_{\leq(\kappa+e_j,\kappa+e_j)}\\
               &= \left((x+u)^{\kappa+e_j} +a\sum_{\beta \leq \kappa} {\kappa \choose \beta}  x^{\beta}u^{\kappa-\beta+e_i}\right)_{\leq(\kappa+e_j,\kappa+e_j)}\\
               &= \left((x+u)^{\kappa+e_j} +a u_i \sum_{\beta \leq \kappa} {\kappa \choose \beta}  x^{\beta}u^{\kappa-\beta}\right)_{\leq(\kappa+e_j,\kappa+e_j)}\\
             &=  \left((x+u)^{\kappa+e_j} + a u_i (x +u)^{\kappa} \right)_{\leq(\kappa+e_j,\kappa+e_j)}    
  \end{align*}
  Now, $(x+u)^{\kappa+e_j} + a u_i (x +u)^{\kappa} = (x_j+u_j+au_i)(x +u)^{\kappa}$ is Lorentzian by Theorem~\ref{thm:changeVar}. Hence, the claim follows by Proposition~\ref{prop:truncation}.  
\end{proof}

\begin{theorem}
  \label{thm:linear-trafo}
  Let $f(x) \in \RR[x_1,\ldots,x_n]$ be dually Lorentzian and $A$ an $(n\times m)$-matrix with non-negative entries, then $f(Ax) \in \RR[x_1, \ldots, x_m]$ is again dually Lorentzian. 
\end{theorem}
\begin{proof}
  We argue as in the proof of \cite[Thm.~2.10]{BrandenHuh}. First, it is easy to see that a \emph{dilation} $f(x_1, \ldots, x_{n-1}, a x_n)$ of a dually Lorentzian polynomial $f$ by an element $a \in \RR_{> 0}$  is again dually Lorentzian. This follows immediately from the corresponding property for Lorentzian polynomials.  For $a=0$ we have $f(x_1, \ldots, x_{n-1}, 0) = (f(x_1, \ldots, x_{n-1}, x_n))_{\leq (d,\ldots,d,0)}$ with $d=\deg(f)$ and this is dually Lorentzian by Corollary~\ref{cor:dual-truncation}.

  Let $f \in \RR[x_1,\ldots,x_n]$ be dually Lorentzian and put $g(x_1,\ldots, x_{n+1})=f(x_1,\ldots, x_n)$. 
  By Lemma~\ref{lem:linear-trafo2} the  polynomial obtained by \emph{elementary splitting}  \[\lim_{k\to \infty}\left(1+\frac{x_{n+1}}{k}\partial_{x_n}\right)^kg= f(x_1, \ldots,x_n+x_{n+1})\] 
  is dually Lorentzian. Likewise,   \[\lim_{k\to \infty}\left(1+\frac{x_{n-1}}{k}\partial_{x_n}\right)^kf= f(x_1, \ldots,x_{n-1}, x_n+x_{n-1})\]
  dually Lorentzian. The \emph{diagonalisation} $f(x_1, \ldots,x_{n-1}, x_{n-1})$ is obtained from the latter polynomial by dilating with $a=0$ along $x_n$. Hence, it is dually Lorentzian.
  The case of a  general linear transformation  follows by induction from dilation, elementary splitting and diagonalisation.  
\end{proof}

\begin{proposition}
  \label{prop:polarization-dual}
  Consider the linear operator \[T \colon \RR_{(\kappa,k)}[x_1,\ldots,x_n,t] \mapsto \RR_{(\kappa,\underline{\mathbf{1}})}[x_1,\ldots,x_n,t_1,\ldots,t_d]\] given by sending $x^\beta t^b$ to $x^\beta e_b(t_1,\ldots,t_k)$, where $e_b$ is the elementary symmetric polynomial of degree $b$. Then $T$ preserves the dual Lorentzian property.
\end{proposition}
\begin{proof}
  For the operator $T^\vee$ we obtain $T^\vee(x^\beta t ^b)=b! \cdot x^\beta e_{k-b}(t_1,\ldots,t_k)$.

  On the other hand, the operator \[\Pi^\uparrow_{d,k} \colon \RR_{(\kappa,k)}[x_1,\ldots,x_n,t] \mapsto \RR_{(\kappa,\underline{\mathbf{1}})}[x_1,\ldots,x_n,t_1,\ldots,t_d]\]
  given by \[x^\beta t^b \mapsto {d \choose b}^{-1} x^\beta e_b(d\cdot t_1,\ldots, d\cdot t_k,0 \ldots,0)\]  preserves homogeneous and stable polynomials by \cite[Lem.~ 2.7,Prob.~3.4]{zbMATH05598918}. Hence, by \cite[Thm.~2.1]{zbMATH05598918} its symbol
  \[
    \operatorname{sym}_{\Pi^\uparrow_{d,k}} =\sum_{\beta,b} { \kappa  \choose \beta}{ k  \choose b} { d  \choose b}^{-1}d^b x^\beta e_b(t_1,\ldots,t_k)
  \]
  is homogeneous and stable and therefore by \cite{BrandenHuh} also Lorentzian. Now, we observe that
  \[\lim_{d\to \infty} \operatorname{sym}_{\Pi^\uparrow_{d,k}} = \sum_{\beta,b} { \kappa  \choose \beta}{ k  \choose b} \cdot b! \cdot x^\beta e_b(t_1,\ldots,t_k)=\operatorname{sym}_{T^\vee}.\]
  By the closedness of the subset of Lorentzian polynomials we conclude that the co-symbol is Lorentzian and therefore $T$ preserves the dual Lorentzian property.
\end{proof}

\begin{corollary}
  \label{cor:polarization-dual}
  Given any dually Lorentzian polynomial $s\in  \RR[x_1,\ldots,x_n]$ the polynomial $\sum s^{(i)}(x_1,\ldots,x_n) e_i(t_1,\ldots,t_k)$ is dually Lorentzian.
\end{corollary}
\begin{proof}
  We have $s(x_1+t, \ldots, x_n+t)=\sum_i s^{(i)}t^i$ is dually Lorentzian by Theorem~\ref{thm:linear-trafo}. Now, applying Proposition~\ref{prop:polarization-dual} gives the desired result.
\end{proof}

For the next result we call that we say that $\mu_0,\ldots, \mu_N$  is a P\'olya frequency sequence
if the matrix $(\mu_{i-j})_{i,j}$ is totally positive (where we set $\mu_i = 0$ for $i < 0$).

\begin{corollary}
  \label{cor:polya-sequence}
  If $\mu_i$ is a P\'olya frequency sequence and $s$ a dually Lorentzian polynomial, then $\sum \mu_i s^{(i)}t^i$ is dually Lorentzian.
\end{corollary}
\begin{proof}
From the Aissen-Schoenberg-Whitney Theorem \cite{zbMATH03077238} we conclude that $\sum_i \mu_i = \lambda \prod_i (t + a_i)$ for some $\lambda >0$ and $a_i \geq 0$, which implies  that $\mu_i=ke_i(a_1\ldots,a_k)$. Now, we consider the dually Lorentzian polynomial $\sum_i s^{(i)}(x_1,\ldots,x_n) e_i(t_1,\ldots,t_k)$ from Corollary~\ref{cor:polarization-dual}. By the linear transformation $t_i \mapsto a_it$ we obtain
  \begin{align*}
    \sum_i s^{(i)}(x_1,\ldots,x_n) e_i(a_1t,\ldots,a_kt)&=\sum_i s^{(i)}(x_1,\ldots,x_n) e_i(a_1,\ldots,a_k)t^i\\
                                                          &=\sum_i \mu_i s^{(i)}(x_1,\ldots,x_n) t^i.
  \end{align*}
  Now, the latter is dually Lorentzian by Theorem~\ref{thm:linear-trafo}.
\end{proof}

\begin{remark} The previous proof follows \cite[Prop.~9.3]{RT20}, but in the general context of dually Lorentzian polynomials and without any reference to geometric constructions. 
\end{remark}

\section{Strictly Lorentzian polynomials}\label{sec:strictly}

We let $H_n^d\subset \mathbb R[x_1,\ldots,x_n]$ denote the subspace of homogeneous polynomials.  Recall that $\mathrm L^d_n\subset \mathrm H^d_n$ denotes the set of Lorentzian polynomials and $\sL^d_n$ denotes the open set of strictly Lorentzian polynomials. 
\begin{lemma}\label{lem:interiorL}
	The interior of $\mathrm L_n^d$ is $\sL^d_n$. 
\end{lemma}	
\begin{proof}
	Let $f$ lie in the interior of $ \mathrm L^d_n$. We will show that  $f\in  \sL^d_n$.
	
	Since $f\in \mathrm L^d_n$, $f$ has non-negative coefficients. Assume that one coefficient of $f$  is in fact zero. Then we could perturb $f$ into a polynomial $\tilde f$ that has a negative coefficient, while staying in interior of $ \mathrm L^d_n$, a contradiction.
	
	The linear map $\partial^\alpha \colon \mathrm H ^d_n \to \mathrm H ^2_n$, $|\alpha|=d-2$, is surjective, hence open. Since $\partial^\alpha \mathrm L^d_n\subset \mathrm L^2_n$,  $\partial^\alpha f$ belongs to the interior of $\mathrm L^2_n$. 
	By the previous paragraph, $\partial^\alpha f $ has positive coefficients. By the Frobenius-Perron theorem, the quadratic form $\partial^\alpha f $ has a positive eigenvalue. If $0$ is an eigenvalue of $\partial^\alpha f $, then we can perturb it into a quadratic form $q$  that lies in the interior of $\mathrm L^2_n$ and  has more than one positive eigenvalue, a contradiction. 
	Thus $\partial^\alpha f \in \overset{\circ}{\mathrm L}{}^2_n$. Since $\alpha$ was arbitrary, we obtain $f\in \sL^d_n$, as desired.
\end{proof}

\begin{lemma}\label{lemma:SurjStrictly}
	Let $T\colon \mathrm H^d_n\to \mathrm H^{d'}_{n'}$, $d'\geq 2$, be a surjective linear map. If $T(\mathrm L^d_n)\subset \mathrm L^{d'}_{n'}$, then $ T(\sL^d_n)\subset \sL^{d'}_{n'}$.
\end{lemma}
\begin{proof}
	Since $T$ is open, this follows from Lemma~\ref{lem:interiorL}
\end{proof}

We need the following elementary fact.
\begin{lemma}\label{lemma:surjectivity}
  Let $\partial \in \RR[\partial_{x_1},\ldots, \partial_{x_n}]$ be a non-zero homogeneous differential operator of degree $d$ and
  $\RR[{x_1},\ldots, {x_n}]_m$ the set of homogeneous polynomials of degree $m$. Then the map
  \[\partial \colon \RR[{x_1},\ldots, {x_n}]_m \to \RR[{x_1},\ldots, {x_n}]_{m-d}\]
  is surjective.
\end{lemma}
\begin{proof}

Since $\partial=\sum c_\alpha \partial^\alpha$ is non-zero, there exists a with respect to lexicographical order a  smallest   multi-index $\alpha$ with $c_\alpha\neq 0$. Let $\beta$ be a multi-index with $|\beta| = m-d$. Observe that 
the largest mononial apperaring  in $\partial x^{\alpha+\beta}$ is  $x^\beta$. 

The claimed surjectivity follows now inductively. Indeed, since $x_n^{m}$ is minimal in lexicographical order, $\partial x^{\alpha+me_n}$ is a non-zero multiple of $x_n^m$.  Let $\beta>me_n$ be a multi-index with $|\beta| = m-d$ and suppose that we have already proved that the image of $\partial \colon \RR[{x_1},\ldots, {x_n}]_m \to \RR[{x_1},\ldots, {x_n}]_{m-d}$ contains all monomials $x^{\beta'}$ with $\beta'<\beta$.  
Since $x^\beta$ is the largest monomial in $\partial x^{\alpha+\beta}$, we conclude that also $x^\beta$ is contained in the image of $\partial$. 
\end{proof}

\begin{theorem}
\label{thm:strictly-lorentzian}
  Let $s\in \RR[x_1,\ldots, x_n]$ be a non-zero dually Lorentzian polynomial of degree $d\leq m$ and let $\partial_s=s(\partial_1,\ldots, \partial_n)\in \RR[\partial_1,\ldots, \partial_n]$. If $f\in  \mathrm  L^m_n$ is strictly Lorentzian, then 
	$ \partial_s f\in \mathrm H^{m-d}_n$ is strictly Lorentzian.
\end{theorem}
\begin{proof} Since $Tf=\partial_s f$ is surjective and preserves Lorentzian polynomials
by Theorem~\ref{thm:operator-criterion},  the claim follows from Lemma~\ref{lemma:SurjStrictly}.
	
\end{proof}

\section{$C$-Lorentzian polynomials}\label{sec:CLorentzian}

Br\"and\'en and Leake \cite{BrandenLeake:CLorentzian} have  recently introduced the notion of $C$-Lorentzian polynomials, where $C$ is an open convex cone in $\RR^n$. In the special case where $C$ is the positive orthant, one recovers the Lorentzian polynomials of Br\"and\'en-Huh. We show here that our main theorem (Theorem~\ref{thm:intropreserve}) extends to this more general setting. We will need this additional  flexibility in the applications.

For $\mathbf v =(v_1,\ldots,v_n) \in \RR^n$ and $f \in \RR[x_1,\ldots,x_n]$ we define the directional derivative $D_{\mathbf{v}} = \sum_i v_i\partial_{x_i} \in \RR[\partial_{x_1},\ldots,\partial_{x_n}]$.

\begin{definition}
\label{def:c-lorentzian}
Let $C \subset \RR^n$ be an open convex cone. A homogeneous polynomial $f \in \RR[x_1,\ldots,x_n]$ of degree $d\geq 1$ is called  \emph{$C$-Lorentzian} if  the following two properties hold:
\begin{enumerate}
	\item  $D_{\mathbf v_1} \cdots D_{\mathbf v_d}f > 0$  for $\mathbf v_1, \ldots, \mathbf v_d \in C$; \label{item:P}
	\item  For $\mathbf v_1, \ldots, \mathbf v_{d-2} \in C$ the bilinear form $$(\mathbf x,\mathbf y) \mapsto  D_{\mathbf x}D_{\mathbf y}D_{\mathbf v_1} \cdots D_{\mathbf v_{d-2}}f$$ has exactly one positive eigenvalue.\label{item:HR}
\end{enumerate}
Moreover,   non-negative constants are defined to be $C$-Lorentzian.
We say that $f$ is \emph{strictly $C$-Lorentzian}, if it lies in the interior of the set of $C$-Lorentzian polynomials.
\end{definition}

Observe that the definition of $C$-Lorentzian polynomials is actually independent of the  choice of coordinates on $\RR^n$. Hence we may speak of $C$-Lorentzian polynomials on a finite-dimensional real vector space, and will use this freely in our applications.

\begin{theorem}\label{thm:invarianceC} Let $C\subset \RR^n$ be an open convex cone and let $\mathbf w_1, \ldots, \mathbf w_m\in C$.
	Suppose that $s\in \mathbb R[x_1,\ldots,x_m]$ is non-zero.  If  $s$ is dually Lorentzian, then the  operator
	$$\partial_s  = s(D_{\mathbf w_1},\ldots,D_{\mathbf w_m}) : \mathbb R[x_1,\ldots,x_n] \to  \mathbb R[x_1,\ldots,x_n]$$
	preserves the property of being (strictly) $C$-Lorentzian.
\end{theorem}

For the proof we need the following facts about $C$-Lorentzian polynomials.

\begin{lemma}[{\cite[Proposition 8.2]{BrandenLeake:CLorentzian}}]\label{lemma:Corthant}
	Let $f\in \RR[x_1,\ldots, x_n]$ be a homogeneous polynomial of degree $d\geq0$. Then $f$ is $\RR^n_{>0}$-Lorentzian if and only if $f$ is Lorentzian.
\end{lemma}

\begin{lemma}[{\cite[Proposition 8.11]{BrandenLeake:CLorentzian}}]\label{lemma:CLorentzian}
	 Let $C\subset \RR^n$ be an open convex cone and let $f\in \RR[x_1,\ldots, x_n]$ be homogeneous of degree $d$. 
	 The following are equivalent:
	 \begin{enumerate}
	 	\item $f$ is $C$-Lorentzian.
	 	\item for all $\mathbf v_1,\ldots, \mathbf v_d\in C$ the polynomial $f(t_1\mathbf v_1+ \cdots + t_d \mathbf v_d)\in \RR[t_1,\ldots, t_d]$ is Lorentzian.
	 \end{enumerate}
\end{lemma}

\begin{lemma}[{\cite[Proposition 2.6]{BrandenLeake:CLorentzian}}]\label{lemma:ACLorentzian} Let $A\colon \RR^m\to \RR^n$ be a linear map. If 
	 $f\in \RR[x_1,\ldots, x_n]$ is $C$-Lorentzian, then $g( x) = f(A x)$ is $A^{-1} C$-Lorentzian.
\end{lemma}

\begin{proof}[Proof of Theorem~\ref{thm:invarianceC}]
	Let $f\in  \RR[x_1,\ldots, x_n]$ be $C$-Lorentzian and let $\mathbf v_1,\ldots, \mathbf v_d\in C$. By Lemma~\ref{lemma:ACLorentzian} the polynomial 
	$$ g(t_1,\ldots, t_{d+m})= f(t_1\mathbf v_1+ \cdots + t_d \mathbf v_d+ t_{d+1}\mathbf w_1+ \cdots + t_{d+m} \mathbf w_m)$$
	is $A^{-1} C$-Lorentzian, where $A$ is the linear map $A( t_1,\ldots, t_{d+m}) = t_1\mathbf v_1+ \cdots + t_d \mathbf v_d+ t_{d+1}\mathbf w_1+ \cdots + t_{d+m} \mathbf w_m$. Since $A^{-1}C$ contains the positive orthant, Lemma~\ref{lemma:Corthant} implies that $g$ is Lorentzian.
	Put  $P= s(\partial_{t_{d+1}},\ldots, \partial_{t_{d+m}})\in \RR[\partial_{t_1},\ldots, \partial_{t_{d+m}}]$ and observe that  
	$$(\partial_s f)(t_1\mathbf v_1+ \cdots + t_d \mathbf v_d) =   P g  ( t_1,\ldots, t_d, 0,\ldots, 0)\in \RR[t_1,\ldots, t_d].$$
	The polynomial on the right-hand side is Lorentzian by Theorem~\ref{thm:operator-criterion}. Thus by Lemma~\ref{lemma:CLorentzian} we have that $\partial_s f$ is $C$-Lorentzian, which proves the theorem for $C$-Lorentzian polynomials.

	Finally just as in the proof of Theorem~\ref{thm:strictly-lorentzian}, we use Lemma~\ref{lemma:surjectivity} to obtain the statement for strictly $C$-Lorentzian polynomials.
\end{proof}

We conclude this section with a characterization of strictly $C$-Lorentzian polynomials.
If $C$ is an open convex cone, we denote by $S_C$ the set of unit vectors which generate extreme rays of the closure of $C$.  The \emph{lineality space} $L$ of $C$ is the largest linear space contained in $\overline{C}$.   We let $S_C/L$ denote set of all vectors in $S_C$ that are orthogonal to $L$. We remark that we will apply the following theorem in Section~\ref{sec:applications} only to cones with trivial lineality space.

\begin{theorem}[{\cite[Theorem 8.14]{BrandenLeake:CLorentzian}}] \label{thm:strictlyClorentzian}
	Let $C$ be an open convex cone in $\RR^n$ with lineality space $L$ and let $f\in \RR[x_1,\ldots, x_n]$ be a homogeneous polynomial of degree $d\geq 2$. Then $f$ is strictly $C$-Lorentzian if and only if the following two properties hold:
	\begin{enumerate}
		\item  $D_{\mathbf v_1} \cdots D_{\mathbf v_d}f > 0$  for $\mathbf v_1, \ldots, \mathbf v_d \in S_C/L$ and  
		\item  for $\mathbf v_1, \ldots, \mathbf v_{d-2} \in S_C/L$ the bilinear form $$(\mathbf x,\mathbf y) \mapsto  D_{\mathbf x}D_{\mathbf y}D_{\mathbf v_1} \cdots D_{\mathbf v_{d-2}}f$$ has exactly one positive  eigenvalue and its radical is $L$.\label{item:HR2}
	\end{enumerate}
\end{theorem}

We will also use the following consequence of the strict $C$-Lorentzian property.
\begin{proposition} \label{prop:strictlyClorentzian}
Let $f\in \RR[x_1,\ldots, x_n]$ be a homogeneous polynomial of degree $d\geq 2$. If $f$ is strictly $C$-Lorentzian, then for all $ \mathbf v_1, \ldots, \mathbf v_{d-2} \in C$ the bilinear form $$(\mathbf x,\mathbf y) \mapsto  D_{\mathbf x}D_{\mathbf y}D_{\mathbf v_1} \cdots D_{\mathbf v_{d-2}}f$$
is non-degenerate and has exactly one positive eigenvalue.
\end{proposition}
\begin{proof}
	The polynomial $g=D_{\mathbf v_1} \cdots D_{\mathbf v_{d-2}}f$ lies in the interior of set of $C$-Lorentzian polynomials. The radical of the quadratic form corresponding  to $g$ must be zero, because otherwise  $g$ could be approximated by a quadratic form with more than one positive eigenvalue, which contradicts the fact that $g$ lies in the interior of the space of $C$-Lorentzian polynomials.  	
\end{proof}

\section{Applications}\label{sec:applications}

\subsection{A generalized AF inequality for mixed discriminants}

Let $V$ a real vector space of dimension $d$. If $q$ is a quadratic form on $V$ its determinant $\det q$ is an element of $(\Lambda^d V^*)^{\otimes 2}$. If $q_1,\ldots, q_m$ are quadratic forms we can expand the determinant as
$$ \det(x_1 q_1+\cdots +x_n q_n) = \sum_\alpha \frac{d!}{\alpha!} D(q^\alpha) x^\alpha$$
where the coefficients 
$$ D(q^\alpha)= D(\underbrace{q_1,\ldots, q_1}_{\alpha_1}, \ldots, \underbrace{q_n, \ldots, q_n}_{\alpha_n})= \frac{1}{d!} \partial^\alpha  \det\Big(\sum x_i q_i\Big)$$  
are known as the \emph{mixed discriminants}.
We may linearly extend this notation by setting
\[
D(f(q)):=\sum_\alpha \lambda_\alpha \cdot D(q^\alpha).
\]
for any homogeneous polynomial $f=\sum_\alpha \lambda_\alpha x^\alpha$ of degree $d$. Further we will write $$D(f(q),g(q)):=D(fg(q))$$ for homogeneous polynomials that factorise.

 We call elements of the form $\omega\otimes \omega\in (\Lambda^d V^*)^{\otimes 2}$ non-negative and  remark that if each $q_i$ is positive definite then $D(q^{\alpha})$ is strictly positive.    Note that if we fix an isomorphism $V\simeq \RR^d $ and identify $\Lambda^d V^*\simeq \RR$,   then  $q(x)= \sum x_ix_j a_{ij}$ can be identified with the matrix $A= (a_{ij})$ and $\det q = \det (A)$, which is positive as a real number if only if $\det(q)\in (\Lambda^d V^*)^{\otimes 2}\simeq \RR$ is positive.

The following is a classical form of the Alexandrov-Fenchel inequality (see, for example \cite[25.4.2]{Burago}).

\begin{theorem}[AF inequality for Mixed Discriminants] Let $V$ be a $d$-dimensional real vector space. 
	If $q_1,\ldots, q_{d-2}$ and  $q$ are positive definite quadratic forms on $V$ and  $p$ is an arbitrary quadratic form, then 
	\begin{equation} D(p,q, q_1,\ldots,q_{d-2})^2 \geq D(p,p, q_1,\ldots,q_{d-2}) D(q,q, q_1,\ldots,q_{d-2}) \label{eq:AFmixed}\end{equation}
	and equality holds if and only if $p$ is proportional to $q$. 
\end{theorem}

\begin{theorem}[Generalized AF inequality for Mixed Discriminants]\label{thm:mixeddiscriminants} Let $V$ be a $d$-dimensional real vector space and $s$ be a non-zero dually Lorentzian polynomial of degree $d-2$ in $n$ variables.  Suppose $q_1,\ldots, q_{n}$, $q$ are positive definite quadratic forms on $V$ and  $p$ is an arbitrary quadratic form.  Then 
	\begin{equation} D(p,q, s(q_1,\ldots,q_n))^2 \geq D(p,p, s(q_1,\ldots,q_n)) D(q,q, s(q_1,\ldots,q_n)) \label{eq:generalAFmixedD}\end{equation}
and  equality holds if and only if $p$ is proportional to $q$. 
\end{theorem}

\begin{proof} 

Let $K\subset \Sym^2(V^*)$ denote the cone of positive definite quadratic forms. Choose an open convex cone $C$ containing $q_1,\ldots, q_n$ with the property that its closure $\overline C$ is contained in $K$. To see the existence of  such a cone $C$ take a full dimensional polytope $P\subset K$  containing $q_1, \ldots, q_n$. Choose an  interior point $v\in \mathrm{int} P$ and let $C$ be the  cone generated $P=  (1+\epsilon)(\mathrm{int} P-v) +v$. By construction $C$ contains $p_1, \ldots, p_n$ and for  $\epsilon> 0$ sufficiently small  $\overline{C}\subset K$.

If $p, q,g_1,\ldots, g_{d-2}$ are  quadratic forms on $V$, then 
\begin{equation}\label{eq:symBil}(p,q)\mapsto  D_p D_q D_{g_1} \cdots D_{g_{d-2}} \det = d! D(p,q, q_1,\ldots, g_{d-2}).\end{equation}
If $g_1,\ldots, g_{d-2}$ are positive definite, then the Alexandrov-Fenchel inequality for mixed discriminants and Lemma~\ref{lemma:hodge} imply that the symmetric billinear form \eqref{eq:symBil} is non-degenerate and  has exactly one positive eigenvalue.
Moreover, the mixed discriminant is positive if all forms are positive definite.   Theorem~\ref{thm:strictlyClorentzian} now implies that the determinant is strictly $C$-Lorentzian for the cone $C$ constructed above.

Put $\partial_s= s(D_{q_1}, \ldots, D_{q_n})$. 
By Theorem~\ref{thm:invarianceC} we have that the polynomial
$$ f(q) = \partial_s \det = \frac{d!}{2} D(q^2, s(q_1,\ldots,q_n))$$
is strictly $C$-Lorentzian.  Proposition~\ref{prop:strictlyClorentzian} implies that the bilinear form 
$$(f,g)\mapsto D(f,g, s(q_1,\ldots,q_n))$$
is non-degenerate and has precisely one positive eigenvalue.  Lemma~\ref{lemma:hodge} yields the desired inequality together with the characterization of equality cases.
\end{proof}

%
%
%

\subsection{A generalized AF inequality for K\"ahler classes}

\begin{theorem}[A generalized AF inequality for K\"ahler classes]
	\label{thm:hodge-riemann}
	Let $Y$ be a smooth complex manifold of dimension $d$ and $\omega_1, \ldots, \omega_m,\alpha$ be K\"ahler classes on $Y$.  Then for every non-zero dually Lorentzian polynomial $s$ of degree $d-2$ and all $\beta\in   H^{1,1}_\RR(Y)$
	$$\left(\int_Y \alpha \beta s(\omega_1,\ldots,\omega_m) \right)^2 \ge \int_Y \alpha^2 s(\omega_1,\ldots,\omega_m)\int_Y \beta^2 s(\omega_1,\ldots,\omega_m)$$
and equality holds if and only if $\alpha$ is proportional to $\beta$.
\end{theorem}
\begin{proof}
Using the mixed version of the Alexandrov-Fenchel inequality for K\"ahler classes \cite{DinhNguyen,Cattani} the same proof as that of Theorem~\ref{thm:mixeddiscriminants} shows that the polynomial
$$ \omega \mapsto  \int_Y \omega^d, \quad \omega\in H_\RR^{1,1}(Y), $$
is strictly Lorentzian with respect to any open convex cone $C$ with the property that its closure $\overline C$ is contained  in the K\"ahler cone, i.e., the open convex  cone  in $H_\RR^{1,1}(Y)$ of K\"ahler classes. 
  Thus the proof is completed essentially the same as that of Theorem \ref{thm:mixeddiscriminants} and is left to the reader.
\end{proof}

\begin{remark}\label{rmk:RT1}
	In the terms of \cite{RT22}, Theorem~\ref{thm:hodge-riemann} (along with the case of equality) says that the
	$2$-cycle $s (\omega_1, \cdots ,\omega_m)$ has the ``Hodge-Riemann property".  When $s$ is a Schur polynomial we recover the main result in \cite{RT22}.
\end{remark}


\begin{remark}\label{rmk:RT2}
	Theorem \ref{thm:hodge-riemann} of course applies to the case that each $\omega_i$ is the first Chern class of an ample line bundle.    Hence, by taking into account Theorem~\ref{thm:schubert},  Proposition~\ref{prop:dual-product}, Corollaries~\ref{cor:derived} and \ref{cor:polya-sequence} we recover a wide range of results about the Hodge-Riemann properties of certain $2$-cycles from \cite{RT19}, \cite{RT20} in a much broader context as they now apply also to Lorentzian polynomials.  
\end{remark}

\begin{remark}\label{rmk:RT3}
	Although we are able to recover (and extend) all of the results from \cite{RT19}, \cite{RT20} concerning collections of ample line bundles, we have not been able to recover the statements concerning ample vector bundles.  In fact it is not clear that this should be possible.  Consider a  Lorentzian polynomial $q$ that is not Schur positive (Example \ref{ex:duallylorentzianbutnotschurpositive}).  Then by \cite[Theorem 8.3.9]{MR2095472}  there exists an ample vector bundle $E$ on a irreducible variety $X$ of dimension $3$ such that $\int_X q(E)<0$.   Thus if $X' = X\times \mathbb P^2$ and $h$ is an ample class on $X'$ sufficiently close to $c_1(\mathcal O_{\mathbb P^2}(1))$ then $\int_{X'} q(E) h^2<0$.  Thus the characteristic classes coming from $q$ are manifestly different from the characteristic classes considered in \cite{RT19,RT20}.
\end{remark}

\subsection{A generalization of the AF inequality for convex bodies}

We now run a similar argument in the context of mixed volumes for convex bodies.  To deal with the case of equality one must restrict attention to non-singular convex bodies, but there is the advantage that the space of such bodies is infinite dimensional and so we do not need to assume any particular upper bound on their number.

To set the scene, let $K_1, \ldots, K_n$ be a collection of convex bodies in dimension $d$. We consider the volume polynomial
$\operatorname{vol}_d(x_1K_1 + \ldots +x_nK_n)$ which has a polynomial expansion
\[
  \vol_d(x_1K_1 + \ldots+ x_nK_n)=\sum_\alpha \frac{d!}{\alpha!}V(K^\alpha)x^\alpha.\]
These normalised coefficients
\[V(K^\alpha)=:V(\underbrace{K_1,\ldots,K_1}_{\alpha_1},\ldots,\underbrace{K_n,\ldots,K_n}_{\alpha_n})=\frac{1}{d!}\partial^\alpha\operatorname{vol}_K\]
are known as \emph{mixed volumes}. We may linearly extend this notation by setting
\[
V(f(K)):=\sum_\alpha \lambda_\alpha \cdot V(K^\alpha).
\]
for any homogeneous polynomial $f=\sum_\alpha \lambda_\alpha x^\alpha$ of degree $d$. Further we will write $$V(f(K),g(K)):=V(fg(K))$$ for homogeneous polynomials that factorise.

Recall that the mixed volumes naturally extend to differences of support functions. Indeed, if $f_i = h_{L_i} - h_{K_i}$, then  we can define $ V(f_1,\ldots, f_n)$  by linear extension.  It is straight forward to check that this is well-defined, and does not depend on how $f_i$ is represented as the difference of two support functions.    
 
 A fundamental inequality between mixed volumes is the Alexandrov-Fenchel inequality discovered in the 1930s.  Despite recent progress \cite{SvH:Extremals,SvH:Polytopes},  a complete characterisation of the equality cases in this inequality  is still missing, but for non-singular bodies $C_1,\ldots, C_{n-2}$ however, the following is known (here we call a convex body $K$ non-singular if through each boundary point of $K$ there exists a unique supporting hyperplane).
 
 \begin{theorem}[{\cite[Theorem 7.6.8]{Schneider:BM}}]\label{thm:AF}
 	If $K,L,C_1,\ldots, C_{n-2}$ are convex bodies in $\RR^n$, then
  	$$ V(K,L, C_1,\ldots, C_{n-2})^2 \geq  V(K,K, C_1,\ldots, C_{n-2}) V(L,L, C_1,\ldots, C_{n-2}).$$ 
  	If the convex bodies $C_1,\ldots, C_{n-2}$ are non-singular, then equality holds if and only if $K$ and $L$ are homothetic. 
 \end{theorem}


%

 A convex body $K$ in $\RR^n$ is said to be of class $C_+^2$ if $\partial K$ is a strictly positively curved $C^2$ submanifold of $\RR^n$.
 
\begin{theorem}[Generalized AF inequality for mixed volumes]\label{thm:genAF}
  Let $K,L,C_1, \ldots, C_n$ be  convex bodies in $\RR^d$. Then  
  \[
    V(K,L,s(C_1, \ldots, C_n))^2 \geq V(K,K,s(C_1, \ldots, C_n))V(L,L,s(C_1, \ldots, C_n))
  \]  
   for every non-zero dually Lorentzian polynomial $s$ of degree $d-2$.
Moreover, if  $K,L,C_1,\ldots, C_{n}$ are of class $C^2_+$, then equality holds if and only if $K$ and $L$ are homothetic.
\end{theorem}
\begin{proof}
	Let $\calH\subset C^2(S^{d-1})$ denote the open convex cone of support functions of convex bodies in the class $C^2_+$. 
	Let $K,L, C_1,\ldots, C_{n}$ 	be convex bodies of class $C_+^2$ and let $E\subset  C^2(S^{d-1})$ denote the subspace spanned by their support functions. Choose an open convex cone $C\subset E$ with $\overline C\subset \calH\cap E$ containing the support functions of the bodies  $K,L, C_1,\ldots, C_{n}$. 
		The Alexandrov-Fenchel inequality,  Lemma~\ref{lemma:hodge}, and Theorem~\ref{thm:strictlyClorentzian} imply that  the volume  polynomial 
	$$ f\mapsto \vol_d(f), \quad f\in E,$$ is strictly $C$-Lorentzian. 
	Thus the proof is completed essentially the same way as that of Theorem \ref{thm:mixeddiscriminants} and is left to the reader. The inequality for general convex bodies follows by approximation.

\end{proof}

\subsection{Valuations on convex bodies} 
Very recently the existence of a K\"ahler package  for valuations on convex bodies has been conjectured. So far Poincar\'e duality and several instances of the hard Lefschetz theorem and Hodge-Riemann bilinear relations have been established \cite{Alesker:Product, BB:Rumin, Kotrbaty:HR, KotrbatyWannerer:AF, KW:HR}.

In the context of convex geometry, a valuation is a function $\phi\colon \calK(\RR^d)\to \RR$ on the space of convex bodies that is finitely additive in the sense that 
$$ \phi(K\cup L) =\phi(K) + \phi(L)-\phi(K\cap L)$$ 
holds whenever $K\cup L$ is convex.  A valuation $\phi$ is called translation-invariant if $\phi(K+x)=\phi(K)$ holds for all $x\in \RR^d$ and all convex bodies $K$ and $\phi$ is said to be homogeneous of degree $\alpha$ if $\phi(tK)= t^\alpha \phi(K)$ holds for all $t>0$. As the space of convex bodies carries a natural topology, one may speak of continuity  of valuations. One of the most important technical innovations introduced by Alesker into the subject is the notion of smoothness of valuations. One possible definition is as follows. Let $\Val=\Val(\RR^d)$ denote the space of translation-invariant and continuous valuation on $\RR^d$. 
 A valuation $\phi\in \Val$ is called smooth if  there is a constant $c\in \RR$ and a smooth differential $(d-1)$-form on the sphere bundle $\RR^n\times S^{d-1}$ that is invariant under translations $(x,u)\to (x+y,u)$ such that 
$$ \phi(K) =  c\vol_d(K) + \int_{N(K)} \omega.$$
Here $N(K)$ denotes the normal cycle of $K$ which as a set consists of the pairs of $(x,u)$ the outward-pointing unit normals $u$ at boundary points $x$ of $K$. 
Let $\Val^\infty\subset \Val$ denote the space of smooth valuations.  By a deep theorem of Alesker~\cite{Alesker:Irr},  smooth valuations form a dense subspace of $\Val$. 

By McMullen's decomposition theorem the spaces $\Val^\infty$ and $\Val$ are graded by the degree of homogeneity of valuations, 
$$ \Val^\infty(\RR^d) = \bigoplus_{i=0}^d \Val_i^\infty(\RR^d).$$
The spaces $\Val_i(\RR^d)$ are infinite dimensional for $0<i<d$. Only in  degrees $0,d-1$, and $d$ these space admit a simple description. $\Val_0$ is spanned by the constant valuation and $\Val_d$ is spanned by volume and so both are one-dimensional. For every $\phi\in \Val_{d-1}$ there exists a continuous function on $S^{d-1}$ that satisfies 
\begin{equation}\label{eq:n-1hom} \phi(K) = V(f,K,\ldots, K)\end{equation}
for all $K$. 
Moreover, up of restrictions of linear functions to the sphere, $f$ is unique. If $\phi$ is smooth, then so is $f$. 
Also the smooth valuations in $\Val_1$ admit a simple description. Namely, if $\phi$ is smooth then there exists a smooth function $f$ on $S^{d-1}$ that is orthogonal to the restriction of linear functions with respect the standard $L^2$ inner product on $S^{d-1}$ and satisfies
\begin{equation}\label{eq:1hom} \phi(K)= \int_{S^{d-1}}  h_K(u)f(u) du\end{equation}
 for all convex bodies $K$. 

With respect to the operation $*$ of convolution valuations, the space $\Val^\infty$ is becomes commutative graded algebra with the volume as identity element, see \cite{BF:Convolution}. 

The generalization of the Alexandrov-Fenchel inequality, Theorem~\ref{thm:genAF}, implies the following statement about smooth valuations.

\begin{theorem}\label{thm:genVal}Let $s$ be a non-zero dually Lorentzian polynomial of degree $d-2$ in $n$ variables. Let $C_0,C_1,\ldots,C_n$ be convex bodies in $\RR^d$ with smooth and strictly positively curved boundary. Let 	$$ \mu_i(K) = V(K,\ldots, K, C_i) $$
	Then the following statements holds:
\begin{enumerate}
	\item Hard Lefschetz theorem: The map $\Val_{d-1}^\infty(\RR^d)\to \Val_{1}^\infty(\RR^d)$, 
	$$ \phi \mapsto \phi * s(\mu_1,\ldots, \mu_n)$$
	is an isomorphism of topological vector spaces.
	\item Hodge-Riemann bilinear relations: For  every $\phi \in \Val_{d-1}^\infty(\RR^d)$
	$$  Q(\phi,\mu_0)^2 \geq Q(\phi,\phi)Q(\mu_0,\mu_0),$$
	where 
	$$ Q(\phi,\mu)= \phi * \mu* s(\mu_1,\ldots, \mu_n) \in \Val_0(\RR^d) \simeq \RR,$$
	and equality holds if and only if $\phi$ is proportional to $\mu_0$. 
\end{enumerate}

\end{theorem}

\begin{remark} Applied to the dually Lorentzian polynomial 
	$$s(x_1,\ldots,x_n) = s_{\lambda^1}(x_1,\ldots,x_{n_1}) \cdots 
	s_{\lambda^p}(x_1,\ldots,x_{n_p}),\quad  n=n_1+ \cdots +n_p,$$
	Theorem~\ref{thm:genVal}, item (2), confirms  Conjecture 4.10 of \cite{hu2023intersection}.
\end{remark}

For the proof of the theorem, let  $g_{S^{d-1}}$ denote the standard Riemannian metric on the  unit sphere $S^{d-1}\subset \RR^n$ and let $\nabla$ be its Levi-Civita connection. It is well-known that if $K_1,\ldots, K_d$ are convex bodies with smooth support functions, then their mixed volume can be expressed as 
$$ V(K_1,\ldots, K_d)= \frac{1}{d} \int_{S^{d-1}} h_{K_d} D^{d-1}( \nabla^2  h_{K_1}  +  h_{K_1} g_{S^{d-1}},\ldots, 
\nabla^2  h_{K_{d-1}}  +  h_{K_{d-1}} g_{S^{d-1}}) $$
where  the mixed discriminant $D^{d-1}$ is defined pointwise using the Riemannian volume form.

 Let $s$ be a dually Lorentzian polynomial of degree $d-2$ in $n$ variables. Let $C_1,\ldots,C_n$ be convex bodies in $\RR^d$ with smooth and strictly positively curved boundary. Let a second-order  linear differential operator 
$ D\colon C^\infty(S^{d-1})\to C^\infty(S^{d-1})$  be defined as follows. 
Let $q_i\in \Sym^2( T^* S^{d-1})$ be 
$$ q_i = \nabla^2  h_{C_i}  +  h_{C_i} g_{S^{d-1}}, \quad i=1,\ldots, n,$$
and put
$$ Df =  D^{d-1}( \nabla^2 f +  f g_{S^{d-1}}, s(q_1, \ldots, q_n)).$$

\begin{lemma}
	$D$ is elliptic and self-adjoint.
\end{lemma}
\begin{proof}
	Fix $p\in S^{n-1}$ and $\xi\in T^*_p S^{n-1}$. Let $f$ be a smooth function on the sphere satisfying $f(p)=0$ and $df_p=\xi$. Since 
	\begin{align*} D(f^2)(p) & = 2 D^{d-1}( \xi\otimes \xi , s(q_1, \ldots, q_n))\\
		& = 2  D^{d-2} ( s(q_1|_{\xi^\perp} ,\ldots, q_n|_{\xi^\perp})) |\xi|^2\\ 
	\end{align*}
is non-zero, if $\xi \neq 0$, ellipticity follows.
Observe that the symmetry of the mixed volumes implies that $D$ is self-adjoint $\langle Df, g\rangle_{L^2} = \langle f, Dg\rangle_{L^2}$.
\end{proof}

\begin{proof}[Proof of Theorem~\ref{thm:genVal}] 
	(2) By \eqref{eq:n-1hom} and the definition of the convolution of valuations, the statement is equivalent to:
	\begin{equation}\label{eq:functionalAF}V(f,C_0, s(C_1,\ldots, C_n))^2  \geq V(f,f, s(C_1,\ldots, C_n))V(C_0,C_0, s(C_1,\ldots, C_n))
	\end{equation}
	holds for all $f\in C^\infty(S^{d-1})$ and equality holds if and only if there exists a constant $a$ such that $f-a h_{C_0}$ is a linear function. 
	
	There exist convex bodies $K_1,K_2$ with a smooth and strictly positively curved boundary such that $f=h_{K_1}-h_{K_2}$. Suppose that $h_{C_0}, h_{K_1}, h_{K_2}$ are linearly independently modulo linear functions. By the generalized AF inequality, Theorem ~\ref{thm:genAF}, the claim follows by applying Lemma~\ref{lemma:hodge} to the quadratic form 
	$$q(x)=V( (x_0C_0+ x_1 K_1 + x_2K_2)^2, s(C_1,\ldots, C_n)), \quad x\in \RR^3.$$
	If  the vectors $h_{C_0}, h_{K_1}, h_{K_2}$ are linearly dependent, we remove $h_{K_1}$ or $h_{K_2}$ and repeat the argument.

(1) As in \cite{KotrbatyWannerer:AF}, because of \eqref{eq:n-1hom} and \eqref{eq:1hom}, it suffices to show that if $g\in C^\infty(S^{d-1})$ is orthogonal to linear functions with respect to the $L^2$ inner product on $S^{n-1}$, then there exists $f\in C^\infty(S^{d-1})$ such that $Df=g$ and  
the kernel of $D$ consists precisely of   linear functions on $\RR^d$ restricted to $S^{d-1}$. 
	
To prove the latter suppose that  $Df=0$. Then  $V(f,C_0,s(C_1,\ldots, C_n))=0$  for every convex body $C_0$ in $\RR^n$. From the characterization of  equality cases in \eqref{eq:functionalAF} we obtain  that $f$ must be the restriction of a linear function on $\RR^d$. 
Conversely, it is straightforward to verify that $\nabla^2 f + fg_{S^{d-1}}=0$ if $f$ is  the restriction of a linear function on $\RR^d$ and hence $Df=0$.

 The proof is completed using the fact that  the self-adjointness and ellipticity of $D$ imply the orthogonal decomposition $C^\infty(S^{d-1})= \operatorname{im} D \oplus \ker D$.

\end{proof}

\subsection{Application to Log Concavity of Derived Polynomials}

We end with some purely combinatorial applications of our main result.

\begin{lemma}
	Let $s\in \RR[x_1,\ldots,x_n]$ be homogeneous of degree $d$.    Then for fixed $t_1,\ldots,t_n\in \mathbb R$,
	\begin{equation}\label{eq:s_operator_derived}
		\partial_s\Big( \big(\sum_i x_i\big)^j \big(\sum_i t_ix_i\big)^{d-j}\Big) =  j!(d-j)!s^{(j)}(t_1,\ldots,t_n).\end{equation}
\end{lemma}
\begin{proof}
	Clearly
	$$\partial_s\Big( \big(\sum_i t_ix_i\big)^d\Big) = d!  s(t_1,\ldots,t_n).$$
	Replacing $t_i$ with $r+t_i$ for $r\in \mathbb R$ yields
	$$\partial_s\Big( \big(\sum_i (t_i+r)x_i\big)^d\Big) = d! s(t_1+r,\ldots,t_n+r) =d! \sum_{j=0}^d s^{(j)}(t_1,\ldots,t_n) r^j.$$
	But
	$$\Big(\sum_i (t_i+r)x_i\Big)^d  = \Big(r \sum_i x_i + \sum_i t_ix_i\Big)^d = \sum_{j=0}^d \binom{d}{j} \Big(\sum_i x_i\Big)^j \Big(\sum_i t_i x_i\Big)^{d-j} r^j.$$
	Using linearity of $\partial_s$ and comparing coefficients of $r^j$ gives \eqref{eq:s_operator_derived}.
\end{proof}

\begin{proposition}\label{prop:duallorentizanlogconcave}
	Let $s\in \RR[x_1,\ldots,x_n]$ be dually Lorentzian of degree $d$ and $t_1,\ldots,t_n\in \mathbb R_{\ge 0}$.  Then the polynomial
	$$ \sum_{j=0}^d \binom{d}{j} s^{(j)}(t_1,\ldots,t_n) u^j v^{d-j}\in \mathbb R[u,v]$$
	is Lorentzian.
	
\end{proposition}
\begin{proof}
	Since each $t_i\ge 0$,
	$$p(x_1,\ldots,x_n,u,v) = \left(\sum_i x_i + v\right)^{d} \left( \sum_i t_i x_i + u\right)^d$$
	is Lorentzian.  We may write this as 
	$$ p = \sum_{j=0}^d \binom{d}{j} \binom{d}{d-j} \left(\sum_i x_i\right)^j \left(\sum_i t_i x_i\right)^{d-j} u^j v^{d-j} + \cdots$$
	where the $\cdots$ are terms that are of degree different to $d$ in $x$.   So from Theorem \ref{thm:operator-criterion} $\partial_s p|_{x_1=\cdots=x_n=0}$ is Lorentzian, which by  \eqref{eq:s_operator_derived} is the statement we want.   \end{proof}

\begin{corollary}\label{cor:duallorentizanlogconcave}
	Let $s\in \RR[x_1,\ldots,x_n]$ be dually Lorentzian of degree $d$ and $t_1,\ldots,t_n\in \mathbb R_{\ge 0}$. Then the map $$j\mapsto s^{(j)}(t_1,\ldots,t_n)$$ for $j=0,\ldots, d$ is log concave.
\end{corollary}
\begin{proof}
	This follows from Proposition \ref{prop:duallorentizanlogconcave} and the fact that a homogeneous polynomial in two variables is Lorentzian if and only if its coefficient sequence is ultra log concave (see \cite[Example 2.3]{BrandenHuh}).
\end{proof}

\begin{remark}
	When $s$ is a Schur polynomial, the statement of Corollary \ref{cor:duallorentizanlogconcave} appears in \cite[Corollary 10.2]{RT20}.
\end{remark}

\begin{proposition}\label{prop:collectionofduallylorentzian}
	Let $s_1,\ldots,s_r\in \RR[x_1,\ldots,x_n]$ be dually Lorentzian of degree $d$.  Suppose that $\underline{t}_j\in \mathbb R^n_{\ge 0}$ for $j=1,\ldots,r$.   Then
	\begin{equation}\label{eq:multischur} \sum_{|\alpha|=d} \prod_{j=1}^r \binom{d}{\alpha_j} s_j^{(\alpha_j)}(\underline{t}_j) u^{\alpha}\in \mathbb R[u_1,\ldots,u_r]\end{equation}
	is Lorentzian.
\end{proposition}
\begin{proof}

	Consider variables $x_{ij}$ for $i=1,\ldots,n$ and $j=1,\ldots,r$ and consider $s_{j}\in \mathbb R[x_{1j},\ldots x_{nj}]$.      Then inside  $\mathbb R[x_{ij},u_1,\ldots,u_r,v]$ the polynomial
	$$p:= \prod_j \Big(\sum_{i} x_{ij} +v\Big)^d \big(\sum_i t_{ij} x_{ij} + u_j\Big)^d$$
	is Lorentzian (using Theorem~\ref{thm:changeVar} which says being Lorentzian is preserved under taking products).
		
	Set $\partial_s = \partial_{s_1}\circ \partial_{s_2} \circ \cdots \circ \partial_{s_r}$ (where we think of $s_j$ as acting purely on the variables $x_{ij}$ for $i=1,\ldots, n$).  Then Theorem \ref{thm:operator-criterion} implies $\partial_s p$ is also Lorentzian, and hence so is 
	$$q : =  \partial_v^{dr-d}  \partial_s p |_{v=0} = \partial_s \partial_v^{dr-d}   p |_{v=0}.$$
	But the coefficient of $v^{dr-d}$  in $p$ is precisely	
	$$\sum_{|\alpha|=d} \prod_{j=1} \binom{d}{\alpha_j}^2 \Big(\sum_i x_{ij}\Big)^{\alpha_j} \Big(\sum_i t_{ij} x_{ij}\Big)^{d-\alpha_j} u^{\alpha_j} + \cdots$$
where the unwritten terms are different from $d$ in $x_{1j},\ldots,x_{rj}$ for some $j$.   	Hence by  \eqref{eq:s_operator_derived} the restriction of $q$ to $x_{ij}=0$ for all $i,j$ is a positive constant times \eqref{eq:multischur}, completing the proof.
\end{proof}

\begin{remark}\label{rmk:RT4}
	In \cite{RT20} the authors ask about the existence of an analog of Proposition \ref{prop:collectionofduallylorentzian} for the derived Schur classes of  a collection of ample vector bundles, and our result here  gives such an analog when all the vector bundles can be written as a direct sum of line bundles. 
\end{remark}


\printbibliography

\end{document}